\documentclass[11pt, a4paper,leqno]{amsart}
\usepackage{amsmath,amsthm,amscd,amssymb,amsfonts, amsbsy}
\usepackage{latexsym}
\usepackage{txfonts}
\usepackage{exscale}
\usepackage{mathrsfs}

\usepackage[colorlinks,citecolor=blue,pagebackref,hypertexnames=false]{hyperref}

\usepackage{pgf}
\usepackage{color}

\calclayout
\allowdisplaybreaks

\theoremstyle{plain}
\newtheorem{theorem}[equation]{Theorem}
\newtheorem{lemma}[equation]{Lemma}
\newtheorem{corollary}[equation]{Corollary}
\newtheorem{proposition}[equation]{Proposition}

\newtheorem{knowntheorem}{Theorem}
\newtheorem{hypothesis}[equation]{Hypothesis}

\theoremstyle{definition}
\newtheorem{definition}[equation]{Definition}

\theoremstyle{remark}
\newtheorem{remark}[equation]{Remark}

\numberwithin{equation}{section}

\newcommand{\eps}{\varepsilon}

\newcommand{\dint}{\int\!\!\!\int}

\newcommand{\dist}{\operatorname{dist}}

\newcommand{\re}{\mathbb{R}}
\newcommand{\rn}{\mathbb{R}^n}

\newcommand{\reu}{\mathbb{R}^{n+1}_+}
\newcommand{\ree}{\mathbb{R}^{n+1}}
\newcommand{\N}{\mathbb{N}}
\newcommand{\dd}{\mathbb{D}}

\newcommand{\C}{\mathcal{C}}

\newcommand{\F}{\mathcal{F}}

\newcommand{\W}{\mathcal{W}}

\newcommand{\I}{\mathcal{I}}

\newcommand{\B}{\mathcal{B}}
\newcommand{\cc}{\mathcal{C}}

\newcommand{\G}{\mathcal{G}}

\newcommand{\mut}{\mathfrak{m}}
\newcommand{\pom}{\partial\Omega}

\newcommand{\hm}{\omega}
\newcommand{\om}{\Omega}

\renewcommand{\emptyset}{\mbox{\textup{\O}}}

\newcommand{\tinyemptyset}{\mbox{\tiny \textup{\O}}}

\DeclareMathOperator{\supp}{supp}

\DeclareMathOperator{\diam}{diam}
\DeclareMathOperator{\interior}{int}

\DeclareMathOperator*{\Lip}{Lip}

\def\div{\mathop{\operatorname{div}}\nolimits}

\def\Xint#1{\mathchoice
   {\XXint\displaystyle\textstyle{#1}}%
   {\XXint\textstyle\scriptstyle{#1}}%
   {\XXint\scriptstyle\scriptscriptstyle{#1}}%
   {\XXint\scriptscriptstyle\scriptscriptstyle{#1}}%
   \!\int}
\def\XXint#1#2#3{{\setbox0=\hbox{$#1{#2#3}{\int}$}
     \vcenter{\hbox{$#2#3$}}\kern-.5\wd0}}
\def\aver#1{\Xint-_{#1}}

\begin{document}
\allowdisplaybreaks

\title{$A_\infty$ implies NTA for a class of variable coefficient elliptic operators}

\author{Steve Hofmann}

\address{Steve Hofmann
\\
Department of Mathematics
\\
University of Missouri
\\
Columbia, MO 65211, USA} \email{hofmanns@missouri.edu}

\author{Jos\'e Mar{\'\i}a Martell}

\address{Jos\'e Mar{\'\i}a Martell
\\
Instituto de Ciencias Matem\'aticas CSIC-UAM-UC3M-UCM
\\
Consejo Superior de Investigaciones Cient{\'\i}ficas
\\
C/ Nicol\'as Cabrera, 13-15
\\
E-28049 Madrid, Spain} \email{chema.martell@icmat.es}

\author{Tatiana Toro}

\address{Tatiana Toro \\ University of Washington \\
Department of Mathematics \\
Seattle, WA 98195-4350, USA}

\email{toro@uw.edu}

\thanks{The first author was supported by NSF grant  DMS-1664047. The second author acknowledges that
the research leading to these results has received funding from the European Research
Council under the European Union's Seventh Framework Programme (FP7/2007-2013)/ ERC
agreement no. 615112 HAPDEGMT. He also acknowledges financial support from the Spanish Ministry of Economy and Competitiveness, through the ``Severo Ochoa Programme for Centres of Excellence in R\&D'' (SEV-2015-0554). 
The last author was partially supported
by the Robert R. \& Elaine F. Phelps Professorship in Mathematics and NSF grant  DMS-1361823.
}

\date{January 4, 2017. \textit{Revised:} \today} 


\begin{abstract}
We consider a certain class of second order, variable coefficient divergence form elliptic operators,
in a uniform domain $\Omega$ with Ahlfors regular boundary,
and we show
that the $A_\infty$ property of the elliptic measure associated to any such operator and its transpose imply that
the domain is in fact NTA (and hence chord-arc).  The converse was already known, and follows from
work of Kenig and Pipher \cite{KP}.
\end{abstract}

\maketitle\setcounter{tocdepth}{3}

\tableofcontents

\section{Main result}

There has been considerable recent activity seeking necessary and sufficient geometric
criteria for
the absolute continuity of harmonic measure with
respect to surface measure, on the boundary of a domain $\Omega$.   This paper is concerned with a
quantitative version of this problem, in which  both the geometric conditions on $\Omega$ and its boundary,
and the absolute
continuity properties of harmonic measure, are expressed in a quantitative, scale invariant fashion.
For example, under the background hypotheses that $\pom$ is Ahlfors regular
(Definition \ref{def1.ADR} below), and that $\Omega$ satisfies
an interior Corkscrew condition (Definition \ref{def1.cork}), in work of the first two authors
\cite{HM4}\footnote{See also \cite{MT} for an alternative proof of the main result of \cite{HM4},
and \cite{HLMN} for an extension of these results to the setting of the $p$-Laplacian.}, it is shown that
$\pom$ is uniformly rectifiable (a quantitative scale invariant version
of rectifiability; see \cite{DS1,DS2}), provided that the Poisson kernel satisfies
a certain uniform scale invariant $L^q$ estimate (which
 is in turn
equivalent to the property that harmonic measure satisfies
 a quantitative, scale-invariant version
of absolute continuity, namely, the weak-$A_\infty$ condition; see
Definition \ref{defAinfty}).


However, the converse to the result
in \cite{HM4} is false
(see \cite{BiJ} for a counter-example), and it remains an open problem to
find a geometric characterization of quantitative absolute continuity of
harmonic measure\footnote{On the other hand, we mention that
certain other boundary estimates for harmonic functions may be characterized in terms of uniform (i.e. quantitative) rectifiability,
by combining the recent results of \cite{HMM} and \cite{GMT}.}.

On the other hand, under strengthened background hypotheses, necessary and sufficient conditions
for absolute continuity are known.  Building on the fundamental result of Dahlberg \cite{Dah} for Lipschitz
domains, David and Jerison \cite{DJe}, and independently Semmes \cite{Se}, showed that
for a chord-arc domain $\Omega$, harmonic measure satisfies an $A_\infty$ condition (see Definition
\ref{defAinfty}) with respect to surface measure $\sigma$ on $\pom$, i.e., harmonic measure is absolutely
continuous with respect to $\sigma$ in a quantitative, scale invariant way.  The term
``chord-arc" refers to an NTA domain with an Ahlfors regular boundary.  In turn, an NTA domain is one
which satisfies the Harnack Chain condition (Definition \ref{def1.hc}), as well as both interior and exterior
Corkscrew conditions.    A domain which satisfies the Harnack Chain condition and
interior Corkscrew condition is known in the literature as a {\it uniform} domain, or a 1-sided NTA domain.
Thus, an NTA domain is a uniform domain for which, in addition,
the exterior Corkscrew condition holds.

In this work, we take as our background hypotheses that $\Omega$ is a uniform (i.e., 1-sided NTA)
domain, and that $\pom$ is Ahlfors regular.  We call such domains {\it 1-sided chord-arc domains}.
In this context, a necessary and sufficient condition for quantitative absolute continuity
(in the form of the $A_\infty$ property) of harmonic measure
with respect to surface measure is known.  Indeed, one has the following.
\begin{knowntheorem}\label{tA}
Suppose that $\om\subset \ree$ is a uniform (aka 1-sided NTA) domain, whose boundary is Ahlfors regular.
Then the following are equivalent:
\begin{enumerate}\itemsep=0.05cm
\item $\pom$ is uniformly rectifiable.
\item $\om$ is an NTA domain, and hence, a chord-arc domain.
\item $\hm\in A_\infty$.
\end{enumerate}
\end{knowntheorem}
As mentioned above, the implication $(2) \implies (3)$ was proved independently in \cite{DJe} and in \cite{Se},
while $(3) \implies (1)$ appears in joint work of the first two
authors of the present paper, together with I. Uriarte-Tuero \cite{HMU}, and $(1) \implies (2)$ was proved
by the present authors jointly with J. Azzam and K. Nystr\"om \cite{AHMNT}.

In the present work, we give a direct (and more efficient) proof of the fact that $(3) \implies (2)$, without passing through
the results of \cite{AHMNT}.  More importantly, our approach here allows us to extend these results well beyond
the case of the Laplacian, and to treat a much broader class of divergence form elliptic operators, with variable
coefficients.    For the class of operators that we consider here, the converse direction $(2) \implies (3)$ was known,
and follows readily from work of Kenig and Pipher \cite{KP};
we shall return to this point momentarily (in particular, see  Corollary \ref{cor1}).

More precisely, we shall consider divergence form elliptic operators
$Lu=-\div(A\nabla u)$, whose coefficients satisfy the following assumptions.

\begin{hypothesis}\label{hyp1}  Let
$A(X)=(a_{i,j}(X))_{1\le i,j\le n+1}$
be a \textbf{real} $(n+1)\times(n+1)$ matrix such that $a_{i,j}\in L^\infty(\Omega)$ for
$1\le i,j\le n+1$, and $A$ is uniformly elliptic, that is, there exists $1\le \Lambda<\infty$ such that 
$$
\Lambda^{-1}\,|\xi|^2\le A(X)\,\xi\cdot\xi,
\qquad
|A(X)\,\xi\cdot\eta|\le 
\Lambda\,|\xi|\,|\eta|,
\qquad
\mbox{for all }\xi,\eta\in\ree,\mbox{ and a.e. }X\in\Omega.
$$
Suppose further that $A$ satisfies the following conditions:
\begin{list}{$(\theenumi)$}{\usecounter{enumi}\leftmargin=.8cm
\labelwidth=.8cm\itemsep=0.2cm\topsep=.1cm
\renewcommand{\theenumi}{\alph{enumi}}}

\item $A\in \Lip_{\rm loc}(\Omega)$.


\item $\big\||\nabla A|\,\delta\big\|_{L^\infty(\Omega)}<\infty$, where
$\delta(X)=\dist(X,\pom)$.

\item $\nabla A$ satisfies the Carleson measure estimate:
\begin{equation}\label{car-A-new}
\|\nabla A\|_{\C(\Omega)}:=\sup_{\substack{x\in\pom\\0<r<\diam(\pom)}} \frac1{\sigma(B(x,r)\cap\pom)}\iint_{B(x,r)\cap\Omega} |\nabla A(X)|dX<\infty,
\end{equation}
\end{list}
where $\sigma:=H^n|_{\pom}$ denotes the surface measure and $H^n$ is the $n$-dimensional Hausdorff measure.
\end{hypothesis}

We shall also assume that $\{\omega_L^X\}_{X\in\Omega}$, the elliptic measure associated with $L$,
belongs to $A_\infty(\pom)$.  More precisely, we shall consider elliptic operators whose associated elliptic measure
satisfies the following scale invariant higher integrability estimate.
\begin{hypothesis}\label{hyp2}
There exists $q>1$ and $C\ge 1$, such that for every
surface ball $\Delta=\Delta(x,r):=B(x,r)\cap\pom$, with $x\in\pom$ and $0<r<\diam(\pom)$,
one has that $\omega_L^{X_{\Delta}}\ll \sigma$,
and the Poisson kernel $k_L^{X_{\Delta}}:=d\omega_L^{X_{\Delta}}/d\sigma$
satisfies the scale invariant estimate
\begin{equation}
\int_{\Delta} k_L^{X_{\Delta}}(y)^q\,d\sigma(y)\le C\,\sigma(\Delta)^{1-q}\,.
\label{eq:higher-inte}
\end{equation}
Here, 
$X_{\Delta}$ is a fixed (or any)
Corkscrew point in $\Omega$, relative to $\Delta$ (see Definition \ref{def1.cork}).
\end{hypothesis}
We remark that in the setting of a uniform domain with Ahlfors regular boundary, it can be shown that estimate \eqref{eq:higher-inte} is in turn
equivalent to the property that the Poisson kernel satisfies an $L^q$ reverse H\"older inequality,
equivalently, that elliptic measure satisfies an
$A_\infty$ condition (a quantitative, scale-invariant version
of absolute continuity; see Definition \ref{defAinfty}).  We further remark that Hypothesis \ref{hyp2} is equivalent to the
solvability of the Dirichlet problem for $L$ with $L^p$ data, and with nontangential maximal function
estimates in $L^p$, for $p=q/(q-1)$ (see, \cite{Ke} and \cite{HMT}).

Our main result is the following. 
\begin{theorem}\label{theor:main}
Let $\Omega\subset\ree$, $n\ge 2$, be a 1-sided CAD (see Definition \ref{defi:CAD}). Let $Lu=-\div (A\,\nabla u)$ be a
second order divergence form operator in $\Omega$, whose coefficients
satisfy the assumptions of Hypothesis \ref{hyp1}, and suppose that the elliptic measures for both $L$ and its transpose
$L^\top$ satisfy Hypothesis \ref{hyp2}.
Then $\Omega$ is a chord-arc domain (CAD).
\end{theorem}

\begin{remark}
In particular, the conclusion that $\om$ is a CAD implies that
$\pom$ is uniformly rectifiable (UR) (see \cite{DS2}). The
UR constants depends on dimension, the ellipticity constants,  the 1-sided CAD constants, $\big\||\nabla A|\,\delta\big\|_\infty$, $\|\nabla A\|_{\C(\Omega)}$ and finally $q$ and $C$ in \eqref{eq:higher-inte}.
\end{remark}

\begin{remark} We point out  that we do not assume symmetry of the coefficient matrix $A$; on the other hand,
Theorem \ref{theor:main} is new even in the symmetric case (and in that case, of course, one requires
that Hypothesis \ref{hyp2} holds only for the elliptic measure associated to $L$).
\end{remark}

Let us observe that given property $(b)$ in Hypothesis \ref{hyp1}, we have that
property $(c)$ immediately implies the more familiar Carleson measure
condition
\begin{equation}\label{car-A-square}
\sup_{\substack{x\in\pom\\0<r<\diam(\pom)}}\frac1{\sigma(B(x,r)\cap\pom)}\iint_{B(x,r)\cap\Omega} |\nabla A(X)|^2 \, \delta(X)\,dX<\infty,
\end{equation}
(on the other hand, the converse is not true).  Here $\delta(X):=\dist(X,\pom)$.
It has essentially been shown by Kenig and Pipher \cite{KP},
that the latter condition, in conjunction with property $(b)$ of
Hypothesis \ref{hyp1},
is sufficient to deduce the $A_\infty$ property of
elliptic measure in an arbitrary Lipschitz domain, and thus, by a well-known maximum principle argument, in
a chord-arc domain $\Omega$
as well, using the method of David and Jerison \cite{DJe}.   We refer the reader to \cite{DJe} for the details, which are stated there
in the case that $L$ is the Laplacian, but extend immediately to the case of operators satisfying property $(b)$ and
\eqref{car-A-square}, since these conditions clearly hold uniformly in every Lipschitz subdomain of $\Omega$.  In particular,
combining our Theorem \ref{theor:main} with the result of \cite{KP} and the method of
\cite{DJe}, we obtain the following necessary and
sufficient criterion.

\begin{corollary}\label{cor1}  Let $\Omega$ be a 1-sided CAD (cf. Definition \ref{defi:CAD}), and
let $Lu=-\div (A\,\nabla u)$ be a variable coefficient second order divergence form operator whose coefficients satisfy
Hypothesis \ref{hyp1}.  Then the elliptic measures $\{\omega_L^X\}_{X\in\Omega}$ and $\{\omega_{L^\top}^X\}_{X\in\Omega}$
belong to $A_\infty(\pom)$, if and only if $\Omega_{ext}:= \ree\setminus \overline{\Omega}$ satisfies the
Corkscrew condition (and thus, $\Omega$ is a CAD).
\end{corollary}

As noted above,  Corollary \ref{cor1} was already known in the special case that $L$ is the Laplacian.

The plan of the paper is as follows. In Section \ref{section:prelim} we present some background and preliminaries that are used throughout the paper. Section \ref{section:aux} contains some auxiliary results. The proof of Theorem \ref{theor:main} is in Section \ref{section:proof-main}. There, after some reductions, we show that it suffices to establish Proposition \ref{prop:CME-G} in the symmetric case and Proposition \ref{prop:CME-G:ns} in general. These in turn follow from an integration by parts argument. In Section \ref{section-sym} we single out the case of symmetric operators and we pay particular attention to the Laplacian since it is rather simple and models the general case which is treated in Section \ref{section-non-sym}.
Finally, we observe that the main theorem of \cite{KKiPT} allows for a modest shortcut in the proof of
the result of \cite{KP} that we have invoked in Corollary \ref{cor1}.
For the reader's convenience, in an appendix
we briefly sketch the argument of \cite{KP}, using the result of \cite{KKiPT} to shorten their proof.

\medskip

\noindent {\it Acknowledgments}.
The main result of this paper was proved in the Fall of 2014 while the first two authors were visiting the last author at
the Mathematics Department of the University of Washington.
In addition,
these results were presented, in condensed form, by the first two authors at the MSRI Summer Graduate School
``Harmonic Analysis and Elliptic Equations
on real
Euclidean Spaces
and on Rough Sets", June 13-24, 2016.  We are grateful to both of these institutions for their kind hospitality.

We mention also that some of the main ideas of this work
were used subsequently (but have appeared already) in \cite{ABHM} to obtain a qualitative version of the main result
of the present paper.
In particular, Proposition \ref{prop:CME-G}, although first proved in the present work,
was previously announced in \cite[Proposition 3.3]{ABHM}, with a
reasonably detailed sketch of the proof\footnote{Quite recently, Proposition \ref{prop:CME-G}, as gleaned from
\cite{ABHM}, has also been used as one ingredient in work of  Azzam, Garnett, Mourgoglou, and Tolsa
\cite{AGMT}, to extend the results of \cite{GMT} to the class of operators with coefficients satisfying
Hypothesis \ref{hyp1}.}.


\section{Preliminaries}\label{section:prelim}

\subsection{Notation and conventions}

\begin{list}{$\bullet$}{\leftmargin=0.4cm  \itemsep=0.2cm}

\item We use the letters $c,C$ to denote harmless positive constants, not necessarily the same at each occurrence, which depend only on dimension and the
constants appearing in the hypotheses of the theorems (which we refer to as the ``allowable parameters'').  We shall also sometimes write $a\lesssim b$ and $a \approx b$ to mean, respectively, that $a \leq C b$ and $0< c \leq a/b\leq C$, where the constants $c$ and $C$ are as above, unless
explicitly noted to the contrary.   Unless otherwise specified upper case constants are greater than $1$  and lower case constants are smaller than $1$. In some occasions it is important to keep track of the dependence on a given parameter $\gamma$, in that case we write $a\lesssim_\gamma b$ or $a\approx_\gamma b$ to emphasize  that the implicit constants in the inequalities depend on $\gamma$.

\item  Our ambient space is $\ree$, $n\ge 2$. 

\item Given $E\subset\ree$ we write $\diam(E)=\sup_{x,y\in E}|x-y|$ to denote its diameter.

\item Given a domain $\Omega \subset \ree$, we shall
use lower case letters $x,y,z$, etc., to denote points on $\partial \Omega$, and capital letters
$X,Y,Z$, etc., to denote generic points in $\ree$ (especially those in $\ree\setminus \partial\Omega$).

\item The open $(n+1)$-dimensional Euclidean ball of radius $r$ will be denoted
$B(x,r)$ when the center $x$ lies on $\partial \Omega$, or $B(X,r)$ when the center
$X \in \ree\setminus \partial\Omega$.  A {\it surface ball} is denoted
$\Delta(x,r):= B(x,r) \cap\partial\Omega$, and unless otherwise specified it is implicitly assumed that $x\in\pom$.

\item If $\pom$ is bounded, it is always understood (unless otherwise specified) that all surface balls have radii controlled by the diameter of $\pom$, that is, if $\Delta=\Delta(x,r)$ then $r\lesssim \diam(\pom)$. Note that in this way $\Delta=\pom$ if $\diam(\pom)<r\lesssim \diam(\pom)$.



\item For $X \in \ree$, we set $\delta(X):= \dist(X,\partial\Omega)$.

\item We let $H^n$ denote the $n$-dimensional Hausdorff measure, and let
$\sigma := H^n|_{\partial\Omega}$ denote the surface measure on $\partial \Omega$.

\item For a Borel set $A\subset \ree$, we let $1_A$ denote the usual
indicator function of $A$, i.e. $1_A(X) = 1$ if $X\in A$, and $1_A(X)= 0$ if $X\notin A$.



\item For a Borel subset $A\subset\partial\Omega$, with $0<\sigma(A)<\infty$, we
set $\aver{A} f d\sigma := \sigma(A)^{-1} \int_A f d\sigma$.

\item We shall use the letter $I$ (and sometimes $J$)
to denote a closed $(n+1)$-dimensional Euclidean cube with sides
parallel to the coordinate axes, and we let $\ell(I)$ denote the side length of $I$.
We use $Q$ to denote  dyadic ``cubes''
on $\partial \Omega$.  The
latter exist, given that $\partial \Omega$ is AR  (cf. \cite{DS1}, \cite{Ch}), and enjoy certain properties
which we enumerate in Lemma \ref{lemma:Christ} below.

\item Given a domain $\om$, and an elliptic operator $L$, we let $\omega_L^X$ denote the $L$-elliptic measure for $\om$
with pole at $X$, and if $\hm_L^X\ll\sigma$, we let $k_L^X:=d\hm_L^X/d\sigma$ be the corresponding Poisson kernel.
When the operator $L$ is understood, we will at times suppress its appearance in the notation, and write simply
$\hm^X$, $k^X$ in place of $\hm_L^X$ and $k_L^X$.
\end{list}

\subsection{Some definitions}\label{ssdefs} 

\begin{definition}[\bf Ahlfors  regular]\label{def1.ADR}
 We say that a closed set $E \subset \ree$ is {\it $n$-dimensional Ahlfors regular} (AR for shortness) if
there is some uniform constant $C$ such that
\begin{equation} \label{eq1.ADR}
C^{-1}\, r^n \leq H^n(E\cap B(x,r)) \leq C\, r^n,\qquad x\in E, \quad 0<r<\diam(E).
\end{equation}
\end{definition}

\begin{definition}[\bf Corkscrew condition]\label{def1.cork}
Following \cite{JK}, we say that a domain $\Omega\subset \ree$
satisfies the {\it Corkscrew condition} if for some uniform constant $c>0$ and
for every surface ball $\Delta:=\Delta(x,r),$ with $x\in \partial\Omega$ and
$0<r<\diam(\partial\Omega)$, there is a ball
$B(X_\Delta,cr)\subset B(x,r)\cap\Omega$.  The point $X_\Delta\subset \Omega$ is called
a {\it corkscrew point relative to} $\Delta$ (or, relative to $B$). We note that  we may allow
$r<C\diam(\pom)$ for any fixed $C$, simply by adjusting the constant $c$.
\end{definition}

\begin{definition}[\bf Harnack Chain condition]\label{def1.hc}
Again following \cite{JK}, we say
that $\Omega$ satisfies the {\it Harnack Chain condition} if there is a uniform constant $C$ such that
for every $\rho >0,\, \Lambda\geq 1$, and every pair of points
$X,X' \in \Omega$ with $\delta(X),\,\delta(X') \geq\rho$ and $|X-X'|<\Lambda\,\rho$, there is a chain of
open balls
$B_1,\dots,B_N \subset \Omega$, $N\leq C(\Lambda)$,
with $X\in B_1,\, X'\in B_N,$ $B_k\cap B_{k+1}\neq \emptyset$
and $C^{-1}\diam (B_k) \leq \dist (B_k,\partial\Omega)\leq C\diam (B_k).$  The chain of balls is called
a {\it Harnack Chain}.
\end{definition}

\begin{definition}[\bf 1-sided NTA and NTA]\label{def1.1nta}
We say that a domain $\Omega$ is a {\it 1-sided NTA  domain} if it satisfies both the Corkscrew and Harnack Chain conditions.
Furthermore, we say that $\Omega$ is an {\it NTA  domain} if it is a 1-sided NTA domain and if, in addition, $\Omega_{\rm ext}:= \ree\setminus \overline{\Omega}$ also satisfies the Corkscrew condition.
\end{definition}
\begin{remark} The abbreviation NTA stands for non-tangentially accessible. In the literature, 1-sided NTA domains are also called \textit{uniform domains}. We remark that the 1-sided NTA condition is a quantitative form of path connectedness.
\end{remark}

\begin{definition}[\bf 1-sided CAD and CAD]\label{defi:CAD}
A \emph{1-sided chord-arc domain} (1-sided CAD) is a 1-sided NTA domain with AR boundary.
A \emph{chord-arc domain} (CAD) is an NTA domain with AR boundary.
\end{definition}

\begin{definition}\label{defAinfty}
({\bf $A_\infty$},  weak-$A_\infty$, and $RH_q$). 
Given an $n$-dimensional Ahlfors regular set $E\subset\ree$, 
and a surface ball
$\Delta_0:= B_0 \cap E$,
we say that a Borel measure $\mu$ defined on $E$ belongs to
$A_\infty(\Delta_0)$ if there are positive constants $C$ and $s$
such that for each surface ball $\Delta = B\cap E$, with $B\subseteq B_0$,
we have
\begin{equation}\label{eq1.ainfty}
\mu (A) \leq C \left(\frac{\sigma(A)}{\sigma(\Delta)}\right)^s\,\mu (\Delta)\,,
\qquad \mbox{for every Borel set } A\subset \Delta\,.
\end{equation}
Similarly, we say that $\mu \in$ weak-$A_\infty(\Delta_0)$ if
for each surface ball $\Delta = B\cap E$, with $2B\subseteq B_0$,
\begin{equation}\label{eq1.wainfty}
\mu (A) \leq C \left(\frac{\sigma(A)}{\sigma(\Delta)}\right)^s\,\mu (2\Delta)\,,
\qquad \mbox{for every Borel set } A\subset \Delta\,.
\end{equation}
We recall that, as is well known, the condition $\mu \in$ $A_\infty(\Delta_0)$
is equivalent to the property that $\mu \ll \sigma$ in $\Delta_0$, and that for some $q>1$, the
Radon-Nikodym derivative $k:= d\mu/d\sigma$ satisfies
the reverse H\"older estimate
\begin{equation}\label{eq1.wRH}
\left(\fint_\Delta k^q d\sigma \right)^{1/q} \,\lesssim\, \fint_{\Delta} k \,d\sigma\,
\approx\,  \frac{\mu(\Delta)}{\sigma(\Delta)}\,,
\quad \forall\, \Delta = B\cap E,\,\, {\rm with} \,\, B\subseteq B_0\,.
\end{equation}
The inequality in \eqref{eq1.wRH} is often referred to as an $L^q$
Reverse H\"older
(``$RH_q$") estimate.  
\end{definition}

%

\subsection{Dyadic grids and sawtooths}
\label{ss:grid}
We first give a lemma concerning the existence of a ``dyadic grid'' which can be found in \cite{DS1,DS2,Ch}.

\begin{lemma}[\bf Existence and properties of the ``dyadic grid'']\label{lemma:Christ}
\cite{DS1,DS2}, \cite{Ch}.
Suppose that $E\subset \ree$ satisfies
the AR condition \eqref{eq1.ADR}.  Then there exist
constants $ a_0>0,\, \eta>0$ and $C_1<\infty$, depending only on dimension and the
AR constant, such that for each $k \in \mathbb{Z},$
there is a collection of Borel sets (``cubes'')
$$
\mathbb{D}_k:=\{Q_{j}^k\subset E: j\in \mathfrak{I}_k\},$$ where
$\mathfrak{I}_k$ denotes some (possibly finite) index set depending on $k$, satisfying

\begin{list}{$(\theenumi)$}{\usecounter{enumi}\leftmargin=.8cm
\labelwidth=.8cm\itemsep=0.2cm\topsep=.1cm
\renewcommand{\theenumi}{\roman{enumi}}}

\item $E=\cup_{j}Q_{j}^k\,\,$ for each
$k\in{\mathbb Z}$.

\item If $m\geq k$ then either $Q_{i}^{m}\subset Q_{j}^{k}$ or
$Q_{i}^{m}\cap Q_{j}^{k}=\emptyset$.

\item For each $(j,k)$ and each $m<k$, there is a unique
$i$ such that $Q_{j}^k\subset Q_{i}^m$.

\item $\diam\big(Q_{j}^k\big)\leq C_12^{-k}$.

\item Each $Q_{j}^k$ contains some surface ball $\Delta \big(x^k_{j},a_02^{-k}\big):=
B\big(x^k_{j},a_02^{-k}\big)\cap E$.

\item $H^n\big(\big\{x\in Q^k_j:{\rm dist}(x,E\setminus Q^k_j)\leq \tau \,2^{-k}\big\}\big)\leq
C_1\,\tau^\eta\,H^n\big(Q^k_j\big),$ for all $k,j$ and for all $\tau\in (0,a_0)$.
\end{list}
\end{lemma}

A few remarks are in order concerning this lemma.

\begin{list}{$\bullet$}{\leftmargin=0.4cm  \itemsep=0.2cm}

\item In the setting of a general space of homogeneous type, this lemma has been proved by Christ
\cite{Ch}, with the
dyadic parameter $1/2$ replaced by some constant $\delta \in (0,1)$.
In fact, one may always take $\delta = 1/2$ (cf.  \cite[Proof of Proposition 2.12]{HMMM}).
In the presence of the Ahlfors regularity property (\ref{eq1.ADR}), the result already appears in \cite{DS1,DS2}.

\item  For our purposes, we may ignore those
$k\in \mathbb{Z}$ such that $2^{-k} \gtrsim {\rm diam}(E)$, in the case that the latter is finite.

\item  We shall denote by  $\mathbb{D}=\mathbb{D}(E)$ the collection of all relevant
$Q^k_j$, i.e., $$\mathbb{D} := \cup_{k} \mathbb{D}_k,$$
where, if $\diam (E)$ is finite, the union runs
over those $k$ such that $2^{-k} \lesssim  {\rm diam}(E)$.

\item For a dyadic cube $Q\in \mathbb{D}_k$, we shall
set $\ell(Q) = 2^{-k}$, and we shall refer to this quantity as the ``length''
of $Q$.  Evidently, $\ell(Q)\approx \diam(Q).$


\item Properties $(iv)$ and $(v)$ imply that for each cube $Q\in\mathbb{D}_k$,
there is a point $x_Q\in E$, a Euclidean ball $B(x_Q,r_Q)$ and a surface ball
$\Delta(x_Q,r_Q):= B(x_Q,r_Q)\cap E$ such that 
\begin{equation}\label{cube-ball}
c\ell(Q)\leq r_Q\leq \ell(Q)
\qquad\mbox{and}\qquad
\Delta(x_Q,2r_Q)\subset Q \subset \Delta(x_Q,Cr_Q),\end{equation}
for some uniform constants $C$, $c$.
We shall denote this ball and surface ball by
\begin{equation}\label{cube-ball2}
B_Q:= B(x_Q,r_Q) \,,\qquad\Delta_Q:= \Delta(x_Q,r_Q),\end{equation}
and we shall refer to the point $x_Q$ as the ``center'' of $Q$.


\end{list}

It will be useful to dyadicize the Corkscrew condition, and to specify
precise Corkscrew constants.
Let us now specialize to the case that  $E=\pom$ is AR,
with $\Omega$ satisfying the Corkscrew condition.
Given $Q\in \mathbb{D}(\partial\Omega)$, we
shall sometimes refer to a ``Corkscrew point relative to $Q$'', which we denote by
$X_Q$, and which we define to be a corkscrew point $X_{\Delta_Q}$ relative to the surface ball
$\Delta_Q$ (see \eqref{cube-ball}, \eqref{cube-ball2} and Definition \ref{def1.cork}).  We note that
\begin{equation}\label{eq1.cork}
\delta(X_Q) \approx \dist(X_Q,Q) \approx \diam(Q).
\end{equation}

\medskip

\begin{definition}[\bf $c_0$-exterior Corkscrew condition]\label{def1.dyadcork}
Fix  a constant $c_0\in (0,1)$,
and a domain $\Omega\subset \ree$, with AR boundary.  We say
that a cube $Q\in \dd(\pom)$ satisfies
the $c_0$-exterior Corkscrew condition, if
there is a point $z_Q\in \Delta_Q$, and a point $X^-_Q\in B(z_Q,r_Q/4) \setminus \overline{\Omega}$, such that
$B(X^-_Q,\,c_0\,\ell(Q))\subset B(z_Q,r_Q/4)\setminus \overline{\Omega}$, where
$\Delta_Q=\Delta(x_Q,r_Q)$ is the surface ball defined above in \eqref{cube-ball}--\eqref{cube-ball2}.
\end{definition}

Following \cite[Section 3]{HM-URHM} we next introduce the notion of \textit{\bf Carleson region} and \textit{\bf discretized sawtooth}.  Given a cube $Q\in\dd(\partial\Omega)$, the \textit{\bf discretized Carleson region $\dd_{Q}$} relative to $Q$ is defined by
\[
\dd_{Q}=\{Q'\in\dd(\partial\Omega):\, \, Q'\subset Q\}.
\]
Let $\F$ be family of disjoint cubes $\{Q_{j}\}\subset\dd(\partial\Omega)$. The \textit{\bf global discretized sawtooth region} relative to $\F$ is the collection of cubes $Q\in\dd$  that are not contained in any $Q_{j}\in\F$;
\[
\dd_{\F}:=\dd\setminus \bigcup\limits_{Q_{j}\in\F}\dd_{Q_{j}}.
\]
For a given $Q\in\dd$ the {\bf local discretized sawtooth region} relative to $\F$ is the collection of cubes in $\dd_{Q}$ that are not in contained in any $Q_{j}\in\F$;
\[
\dd_{\F,Q}:=\dd_{Q}\setminus \bigcup\limits_{Q_{j}\in \F} \dd_{Q_{j}}=\dd_{\F}\cap \dd_{Q}.
\]
We also introduce the ``geometric'' Carleson regions and sawtooths. In the sequel, $\Omega \subset \ree$ ($n\geq 2$) will be a 1-sided CAD domain. Let $\mathcal{W}=\W(\Omega)$ denote a collection
of (closed) dyadic Whitney cubes   of $\Omega$ (see \cite[Chapter VI]{St}),  so that the boxes  in $\mathcal{W}$
form a covering of $\Omega$ with non-overlapping interiors, and  which satisfy
\begin{equation}\label{eqWh1} 4\, {\rm{diam}}\,(I)\leq \dist(4 I,\pom) \leq  \dist(I,\pom) \leq 40 \, {\rm{diam}}\,(I)\end{equation}
and
\begin{equation}\label{eqWh2}\diam(I_1)\approx \diam(I_2), \mbox{ whenever $I_1$ and $I_2$ touch.}\end{equation}
Let $X(I)$ denote the center of $I$, let $\ell(I)$ denote the side length of $I$,
and write $k=k_I$ if $\ell(I) = 2^{-k}$. We will use ``boxes'' to refer to the Whitney cubes as just constructed, and ``cubes'' for the dyadic cubes on $\pom$. 

Given $0<\lambda<1$ and $I\in\W$ we write $I^*=(1+\lambda)I$ for the ``fattening'' of $I$. By taking $\lambda$ small enough,  we can arrange matters so that, first, $\dist(I^*,J^*) \approx \dist(I,J)$ for every
$I,J\in\W$, and, secondly, $I^*$ meets $J^*$ if and only if $\partial I$ meets $\partial J$.
 (Fattening ensures $I^*$ and $J^*$ overlap for
any pair $I,J \in\W$ whose boundaries touch. Thus, the Harnack Chain property holds locally in $I^*\cup J^*$ with constants depending on $\lambda$.)  By picking $\lambda$ sufficiently small, say $0<\lambda<\lambda_0$, we may also suppose that there is $\tau\in(1/2,1)$ such that for distinct $I,J\in\W$, $\tau J\cap I^* = \emptyset$. In what follows we will need to work with dilations $I^{**}=(1+2\,\lambda)I$ or $I^{***}=(1+4\,\lambda)I$ and in order to ensure that the same properties hold we further assume that $0<\lambda<\lambda_0/4$.

For every $Q$ we can construct a non-empty family $\W_Q^*\subset \W$ and define
\begin{equation}\label{eq2.whitney3}
U_Q := \bigcup_{I\in\,\mathcal{W}^*_Q} I^*\,,
\end{equation}
satisfying the following properties:
$X_Q\in U_Q$ and there are uniform constants $k^*$ and $K_0$ such that
\begin{align}\label{eq2.whitney2}
\begin{array}{cl}
 k(Q)-k^*\leq k_I \leq k(Q) + k^*\, \quad & \forall\, I\in \mathcal{W}^*_Q,
\\[5pt]
X(I) \rightarrow_{U_Q} X_Q\,\quad &\forall\, I\in \mathcal{W}^*_Q,
\\ [5pt]
\dist(I,Q)\leq K_0\,2^{-k(Q)}\, \quad &\forall\, I\in \mathcal{W}^*_Q\,.
\end{array}
\end{align}
Here $X(I) \rightarrow_{U_Q} X_Q$ means that the interior of $U_Q$ contains all the balls in
a Harnack Chain (in $\Omega$) connecting $X(I)$ to $X_Q$, and  moreover, for any point $Z$ contained
in any ball in the Harnack Chain, we have $
\dist(Z,\pom) \approx \dist(Z,\Omega\setminus U_Q)$
with uniform control of the implicit constants.
The constants  $k^*$, $K_0$ and the implicit constants in the condition $X(I)\to_{U_Q} X_Q$ in \eqref{eq2.whitney2}
depend on at most allowable
parameters and on $\lambda$. For later use, it will be convenient to associate to Whitney boxes a particular nearest dyadic cube. Let $I\in\W$ with $\ell(I)\lesssim\diam(\pom)$ and pick $z\in\pom$ (there could be more than one $z$ with this property but we just pick one) such that $\dist(I,\pom)=\dist(I,z)$. We define $Q_I^*\in\dd$ as the unique dyadic cube such that $z\in Q_I^*$ with $\ell(Q_I^*)=\ell(I)$. We note that the construction in \cite{HM-URHM} guarantees that $I\in \W^*_{Q_I^*}$ (indeed, this property holds for any other nearest dyadic cube with side length $\ell(I)$).
The reader is referred to \cite{HM-URHM} for full details.

For a given $Q\in\dd$, the {\bf Carleson box} relative to $Q$ is defined by
\[
T_{Q}:=\mbox{int}\left(\bigcup\limits_{Q'\in\dd_{Q}} U_{Q'}\right).
\]
For a given family $\F$ of disjoint cubes $\{Q_{j}\}\subset\dd$ and a given $Q\in\dd$ we define the {\bf local sawtooth region} relative to $\F$ by
\[
\Omega_{\F,Q}:=\mbox{int}\left(\bigcup\limits_{Q'\in\dd_{\F,Q}}U_{Q'}\right)
=
{\rm int }\,\left(\bigcup_{I\in\,\W_{\F,Q}} I^*\right),
\]
where $\W_{\F,Q}:=\bigcup_{Q'\in\dd_{\F,Q}}\W^*_{Q'}$.
Analogously, we can slightly fatten the Whitney boxes and use $I^{**}$ to define new fattened Whitney regions and sawtooth domains. More precisely,
\begin{equation}
T_{Q}^*:=\mbox{int}\left(\bigcup\limits_{Q'\in\dd_{Q}} U_{Q'}^*\right),
\qquad
\Omega_{\F,Q}^*:=\mbox{int}\left(\bigcup\limits_{Q'\in\dd_{\F,Q}}U_{Q'}^*\right)
,
\qquad
U_{Q'}^*:=\bigcup_{I\in\,\mathcal{W}^*_{Q'}} I^{**}.
\label{eq:fatten-objects}
\end{equation}
Similarly, we can define $T_Q^{**}$, $\Omega_{\F,Q}^{**}$ and $U_{Q}^{**}$ by using $I^{***}$ in place of $I^{**}$.

One can easily see that there is $\kappa_0>c^{-1}$ (depending only on the allowable parameters and where $c$ is the constant in \eqref{cube-ball})  so that 
\begin{equation}\label{tent-Q}
T_Q\subset T_{Q}^*\subset T_{Q}^{**}\subset\overline{T_Q^{**}}\subset \kappa_0 B_Q\cap\overline{\Omega}
=: B_Q^*\cap \overline{\Omega}, \qquad\forall\,Q\in\dd.
\end{equation}

Given a pairwise disjoint family $\F\subset\dd$ (we also allow $\F$ to be the null set) and a constant $\rho>0$, we derive another family $\F({\rho})\subset\dd$  from $\F$ as follows. Augment $\F$ by adding cubes $Q\in\dd$ whose side length $\ell(Q)\leq \rho$ and let $\F(\rho)$ denote the corresponding collection of maximal cubes with respect to the inclusion. Note that the corresponding discrete sawtooth region $\dd_{\F(\rho)}$ is the union of all cubes $Q\in\dd_{\F}$ such that $\ell(Q)>\rho$.
For a given constant $\rho$ and a cube $Q\in \dd$, let $\dd_{\F(\rho),Q}$ denote the local discrete sawtooth region and let $\Omega_{\F(\rho),Q}$ denote the local geometric sawtooth region relative to disjoint family $\F(\rho)$.

Given $Q\in\dd$ and $0<\epsilon<1$, if we take $\F_0=\emptyset$, one has that \
$\F_0(\epsilon \,\ell(Q))$ is the collection of $Q'\in \dd$ such that $\epsilon\,\ell(Q)/2<\ell(Q')\le \epsilon\,\ell(Q)$. We then introduce $U_{Q,\epsilon}=\Omega_{\F_0(\epsilon \,\ell(Q)),Q}$, which is a Whitney region relative to $Q$ whose distance to $\pom$ is of the order of $\epsilon\,\ell(Q)$.  For later use, we observe that given $Q_0\in\dd$, the sets $\{U_{Q,\epsilon}\}_{Q\in\dd_{Q_0}}$ have bounded overlap with constant that may depend on $\epsilon$. Indeed, suppose that there is $X\in U_{Q,\epsilon}\cap U_{Q',\epsilon}$ with $Q$, $Q'\in \dd_{Q_0}$. By construction $\ell(Q)\approx_\epsilon \delta(X)\approx_\epsilon \ell(Q')$ and
$$
\dist(Q,Q')\le \dist(X,Q)+\dist(X,Q')\lesssim_\epsilon \ell(Q)+\ell(Q')\approx_\epsilon \ell(Q).
$$
The bounded overlap property follows then at once.

\subsection{PDE estimates}

Next, we recall several facts concerning elliptic measure and Green's functions. For our first results we will only assume that $\Omega\subset\ree$, $n\ge 2$, is an open set, not necessarily connected with $\pom$ being AR. Later we will focus on the case where
$\Omega$ is 1-sided CAD.

Let $Lu=-\div (A\,\nabla u)$ be a variable coefficient second order divergence form operator with $A(X)=(a_{i,j}(X))_{1\le i,j\le n+1}$ being a real (non-necessarily symmetric) $(n+1)\times(n+1)$ matrix such that $a_{i,j}\in L^\infty(\Omega)$ for $1\le i,j\le n+1$, and $A$ is uniformly elliptic, that is, there exists $1\le \Lambda<\infty$ such that
$$
\Lambda^{-1}\,|\xi|^2\le A(X)\,\xi\cdot\xi,
\qquad
|A(X)\,\xi\cdot\eta|\le 
\Lambda\,|\xi|\,|\eta|,
$$
for all $\xi, \eta \in \re^{n+1}$ and almost every $X\in \Omega$. We write $L^\top$ to denote the transpose  of $L$, or, in other words,  $L^\top u=-\div (A^\top\,\nabla u)$ with $A^\top$ being the transpose matrix of $A$.

We say that a function $u\in W^{1,2}_{\rm loc}(\Omega)$ is a weak solution to $Lu=0$ in $\Omega$ or that $Lu=0$ in the weak sense in $\Omega$, if
$$
\iint_{\Omega} A(X)\,\nabla u(X)\cdot \nabla\Phi(X)\,dX=0,
\qquad
\forall\,\Phi\in C_0^\infty(\Omega).
$$

Associated with $L$ and $L^\top$ one can respectively construct the elliptic measures $\{\omega_L^X\}_{X\in \Omega}$ and $\{\omega_{L^\top}^X\}_{X\in\Omega}$, and the Green functions $G_L$ and $G_{L^\top}$. The construction of the Green functions dates back to Gruter and Widman \cite{GW}, while the existence of the corresponding elliptic measures
is an application of the Riesz representation theorem in the case the domain is Wiener regular. The behavior of $\omega^X_L$ (resp. $\omega^X_{L^T}$) and $G_L$ (resp. $G_{L^T}$) as well as the relationship between them is something that depends crucially on the fact that $\Omega$ is a 1-sided CAD. For a comprehensive treatment of the subject we refer the reader to the forthcoming monograph \cite{HMT} (a summary of some of these properties can also be found in \cite{Z}).

\begin{lemma}\label{Bourgainhm}  Suppose that
$\partial \Omega$ is $n$-dimensional AR.  Then there are uniform constants $c\in(0,1)$
and $C\in (1,\infty)$, depending only on $n$, AR, and $\Lambda$
such that for every $x \in \partial\Omega$, and every $r\in (0,\diam(\partial\Omega))$,
if $Y \in \Omega \cap B(x,cr),$ then
\begin{equation}\label{eq2.Bourgain1}
\omega_L^{Y} (\Delta(x,r)) \geq 1/C>0 \;.
\end{equation}
\end{lemma}
We refer the reader to \cite[Lemma 1]{B} for the proof in the harmonic case and to \cite{HMT} for general elliptic operators.  See also \cite[Theorem 6.18]{HKM} and \cite[Section 3]{Z}.

The next result incorporates the construction of the Green function in \cite{GW} with some of the properties which are derived from the assumptions on the domain (all details can be found in \cite{HMT}). We note that, in particular, the AR hypothesis implies that $\pom$ is Wiener regular at every point. In fact, it satisfies the Capacity Density Condition (CDC) (see 
\cite[Lemma 3.27]{HLMN} and \cite[Section 3]{Z}).

\begin{lemma} 
 \label{lemma2.green}
Let $\Omega$ be an open set with $n$-dimensional AR boundary. Let $Lu=-\div (A\,\nabla u)$ be as above. There are positive, finite constants $C$, depending only on dimension, $\Lambda$
and $c_\theta$, depending on dimension, $\Lambda$, and $\theta \in (0,1),$
such that $G_L$, the Green function associated with $L$, satisfies
\begin{equation}\label{eq2.green}
G_L(X,Y) \leq C\,|X-Y|^{1-n}\,
\end{equation}
\begin{equation}\label{eq2.green2}
c_\theta\,|X-Y|^{1-n}\leq G_L(X,Y)\,,\quad {\rm if } \,\,\,|X-Y|\leq \theta\, \delta(X)\,, \,\, \theta \in (0,1)\,;
\end{equation}
\begin{equation}
\label{eq2.green-cont}
G_L(\cdot,Y)\in C(\overline{\Omega}\setminus\{Y\}) \qquad \mbox{and}\qquad G_L(\cdot,Y)\big|_{\pom}\equiv 0\,,\qquad \forall Y\in\Omega;
\end{equation}
\begin{equation}
\label{eq2.green3}
G_L(X,Y)\geq 0\,,\qquad \forall X,Y\in\Omega\,,\, X\neq Y;
\end{equation}
\begin{equation}\label{eq2.green4}
G_L(X,Y)=G_{L^\top}(Y,X)\,,\qquad \forall X,Y\in\Omega\,,\, X\neq Y.
\end{equation}
Moreover,  $G_L(\cdot, Y)\in W^{1,2}_{\rm loc}(\Omega\setminus \{Y\})$ for any $Y\in\Omega$ and satisfies
$L G_L(\cdot,Y)=\delta_Y$ in the weak sense in $\Omega$, that is,
$$
\iint_\Omega A(X)\,\nabla_X G_{L}(X,Y) \cdot\nabla\Phi(X)\, dX=\Phi(Y), \qquad
\forall\,\Phi \in C_0^\infty(\Omega).
$$
In particular, $G_L(\cdot,Y)$ is a weak solution to $L G_L(\cdot,Y)=0$ in the open set $\Omega\setminus\{Y\}$.

Finally, the following Riesz formula holds
\begin{equation}\label{eq2.14}
\iint_\Omega
A^\top(Y)\,\nabla_Y G_{L^\top}(Y,X) \cdot\nabla\Phi(Y)\, dY
=
\Phi(X)-\int_{\partial\Omega} \Phi\,d\omega_L^X,
\qquad
\mbox{for a.e. } X\in\Omega,
\end{equation}
and for every $\Phi \in C_0^\infty(\ree)$.
\end{lemma}

Next, we recall a Caffarelli-Fabes-Mortola-Salsa estimate (cf. \cite{CFMS}, and \cite{HMT} for the current version).

\begin{lemma} \label{lemma2.cfms}
Let $\Omega$ be a 1-sided CAD domain.
Let $B:=B(x,r)$, with $x\in \pom$, $0<r<\diam(\pom)$. Then for
$X\in \Omega\setminus 2\,B$ we have
\begin{equation}\label{eq.CFMS}
\frac1C\omega_L^X(\Delta)
\leq
r^{n-1}G_L(X,X_\Delta) \leq C \omega_L^X(\Delta).
\end{equation}
The constant in \eqref{eq.CFMS} depends {\bf only} on $\Lambda$,
dimension and on the constants in the 1-sided CAD character.
\end{lemma}

\begin{lemma}\label{lemma.double}
Suppose that $\Omega$ is a  1-sided CAD domain. Let $B:=B(x,r)$,  $x\in \pom$,
$\Delta:= B\cap\partial\Omega$ and $X\in \Omega\setminus 4B.$
Then there is a uniform constant $C$, depending {\bf only} on $\Lambda$,
dimension and on the constants in the 1-sided CAD character, such that
\begin{equation}
\omega_L^X(2\Delta)\leq C\omega_L^X(\Delta).
\label{eq:elliptic-doubling}
\end{equation}
\end{lemma}

\section{Auxiliary results}\label{section:aux}

We have the following Poincar\'e inequality which is an improvement of  \cite[Lemma 4.8]{HM-URHM}.

\begin{lemma}\label{poincare}
Suppose that $\Omega$ is a 1-CAD.
Fix $Q_0\in \dd$, a (possibly empty) pairwise disjoint family $\F\subset \dd_{Q_0},$ and let
$Q\in \dd_{\F,Q_0}$. Then for every $p,\, 1\leq p<\infty$, and for every small $\epsilon >0$,
there is a constant $C_{\epsilon,p}$ such that
\begin{equation}\label{p-eq}
\iint_{\Omega_{\F(\epsilon\,\ell(Q)),Q}}|f(X)-c_{Q,\epsilon}|^p\, dX \leq C_{\epsilon,p}\,
\ell(Q)^p\iint_{\Omega_{\F(\epsilon \,\ell(Q)),Q}}|\nabla f(X)|^p\,dX,
\end{equation}
where $c_{Q,\epsilon}:= |\Omega_{\F(\epsilon \,\ell(Q)),Q}|^{-1}\iint_{\Omega_{\F(\epsilon\,\ell(Q)),Q}}f\,dX.$ In particular, the previous Poincar\'e inequality holds for $U_{Q,\epsilon}$ replacing $\Omega_{\F(\epsilon \,\ell(Q)),Q}$.
\end{lemma}

\begin{proof}
Without loss of generality we may assume that $\Omega_{\F(\epsilon\,\ell(Q)),Q}\neq\emptyset$.
We first observe that
\begin{multline}\label{poin-aux1}
\iint_{\Omega_{\F(\epsilon\,\ell(Q)),Q}}
\left|f(X)\,-\,\frac1{|\Omega_{\F(\epsilon \,\ell(Q)),Q}|}\iint_{\Omega_{\F(\epsilon \,\ell(Q)),Q}}f(Y)\,dY\right|^pdX
\\
\le
\frac1{|\Omega_{\F(\epsilon \,\ell(Q)),Q}|} \iint_{\Omega_{\F(\epsilon\,\ell(Q)),Q}} \iint_{\Omega_{\F(\epsilon\,\ell(Q)),Q}} |f(X)-f(Y)|^p\,dX\,dY
\\
\le
\frac1{|\Omega_{\F(\epsilon \,\ell(Q)),Q}|}
\sum_{I, J\in\ \W_{\F(\epsilon \,\ell(Q)),Q}} \iint_{I^*}\iint_{J^*} |f(X)-f(Y)|^p\,dX\,dY.
\end{multline}
Fix now $I, J\in\ \W_{\F(\epsilon \,\ell(Q)),Q}$. Note that $\dist(I,J)\lesssim \ell(I)\approx \ell(J)\approx \ell(Q)$  (where the implicit constants depend upon $\epsilon$). By \cite[Lemma 3.61]{HM-URHM}
there is a chain $\{I_1,I_2,\dots,I_N\}\subset\W_{\F(\epsilon \,\ell(Q)),Q}$,
of bounded cardinality $N$ depending only on dimension, the 1-sided CAD constants of $\Omega$, and $\epsilon$, such that
$I_1=J,\, I_N=I$, $\ell(I_j)\approx\ell(I)\approx\ell(J)$ for each $j$
(again the implicit constants depend upon $\epsilon$),  and for which
$\cup_{j=1}^NI^*_j$ contains a Harnack Chain
which connects the centers of $I$ and $J$.  Moreover,
the chain may be constructed so that $I^*_j\cap I^*_{j+1}\neq\emptyset$, $1\le j\le N-1$. Hence, by telescoping and using the standard Poincar\'e inequality
\begin{align}\label{po-tele}
&\left(\iint_{I^*}\iint_{J^*} |f(X)-f(Y)|^p\,dX\,dY\right)^\frac1p
\\ \nonumber
&\ \lesssim
\ell(I)^{\frac{n+1}p}\left(\left(\iint_{I^*} |f(X)-f_{I^*}|^p\,dX\right)^\frac1p
+
\left(\iint_{J^*} |f(Y)-f_{J^*}|^p\,dY\right)^\frac1p\right)
+
\ell(I)^{\frac{2\,(n+1)}p} \sum_{j=1}^{N-1} |f_{I^*_j} - f_{I^*_{j+1}}|
\\ \nonumber
&\  \lesssim
\ell(I)^{\frac{n+1}p+1}\left(\iint_{\Omega_{\F(\epsilon \,\ell(Q)),Q}} |\nabla f(X)|^p\,dX\right)^\frac1p
+
\ell(I)^{\frac{2\,(n+1)}p} \sum_{j=1}^{N-1} |f_{I^*_j} - f_{I^*_{j+1}}|,
\end{align}
where we have used the notation  $f_{I^*_j}:= |I^*_j|^{-1}\dint_{I^*_j} f$, and where the implicit constants depend on $p$ and $\epsilon$.
To analyze the last term, take $1\le j\le {N-1}$. Recall that $I_j^*=(1+\lambda)\,I_j$ with $I_j$ being a dyadic Whitney cube. The same applies to  $I_{j+1}^*$. Also, by choice of $\lambda$ and since $I^*_j\cap I_{j+1}^*\neq\emptyset$ it follows that $\partial I_j$  meets $\partial I_{j+1}$, which in turns implies that $\ell(I_j)\approx \ell(I_{j+1})$. Hence, one can find a cube $\tilde{I}\subset I^*_j\cap I^*_{j+1}$ with $\ell(\tilde{I})\approx \lambda\ell(I_j)\approx \lambda\ell(I_{j+1})$.
Then, by using again the standard Poincar\'e inequality we conclude
\begin{align*}
|f_{I^*_j} - f_{I^*_{j+1}}|
&\le
|f_{I^*_j}-f_{\tilde{I}}|+|f_{\tilde{I}}- f_{I^*_{j+1}}|
\le
\frac1{|\tilde{I}|}\iint_{\tilde{I}} |f(X)-f_{I^*_j}|\,dX
+
\frac1{|\tilde{I}|}\iint_{\tilde{I}} |f(X)-f_{I^*_{j+1}}|\,dX
\\
&\lesssim
\frac1{|I^*_{j}|}\iint_{I^*_{j}} |f(X)-f_{I^*_j}|\,dX
+
\frac1{|I^*_{j+1}|}\iint_{I^*_{j+1}} |f(X)-f_{I^*_{j+1}}|\,dX
\\
&\lesssim
\frac{\ell(I^*_{j})}{|I^*_{j}|}\iint_{I^*_{j}} |\nabla f(X)|\,dX
+
\frac{\ell(I^*_{j+1})}{|I^*_{j+1}|}\iint_{I^*_{j+1}} |\nabla f(X)|\,dX
\end{align*}
where the implicit constants depend on $n$ and $\lambda$. Now, we plug this estimate into \eqref{po-tele}, use that
$\ell(I_j^*)\approx \ell(I)\approx\ell(Q)$ (with constants that depend on $\epsilon$), the bounded overlap of the family $\{I_j^*\}_{j=1}^N$  and that $N$ depends upon $\epsilon$ (it also depends on $I$ and $J$,  but in a uniformly bounded manner for $\epsilon$ fixed):
\begin{align*}
&\left(\iint_{I^*}\iint_{J^*} |f(X)-f(Y)|^p\,dX\,dY\right)^\frac1p
\\
&\quad\lesssim
\ell(I)^{\frac{n+1}p+1}\left(\iint_{\Omega_{\F(\epsilon \,\ell(Q)),Q}} |\nabla f(X)|^p\,dX\right)^\frac1p
+
\ell(I)^{\frac{n+1}p+1-\frac{n+1}{p'}}
\iint_{\cup_{j=1}^N I_j^*} |\nabla f(Y)|\,dY
\\
&\quad\lesssim
\ell(I)^{\frac{n+1}p+1}\left(\iint_{\Omega_{\F(\epsilon \,\ell(Q)),Q}} |\nabla f(X)|^p\,dX\right)^\frac1p
+
\ell(I)^{\frac{n+1}p+1-\frac{n+1}{p'}} \left|\cup_{j=1}^N I_j^*\right|^{\frac{1}{p'}}
\left(\iint_{\cup_{j=1}^N I_j^*} |\nabla f(Y)|^p\,dY\right)^{\frac{1}{p}}
\\
&\quad\lesssim
|\Omega_{\F(\epsilon \,\ell(Q)),Q}|^{\frac1p}\,\ell(Q)\,\left(\iint_{\Omega_{\F(\epsilon \,\ell(Q)),Q}} |\nabla f(Y)|^p\,dY\right)^\frac1p.
\end{align*}
Both in the last line above, and in order to conclude the proof of \eqref{p-eq} from this inequality  and  \eqref{poin-aux1}, we need to observe that
$$
\#\W_{\F(\epsilon \,\ell(Q)),Q}
\le
\#\big\{Q'\in\dd_Q: \ell(Q')>\epsilon\,\ell(Q)\big\}
\le
C_\epsilon.
$$
\end{proof}

The following result is of purely real variable nature and establishes that if a measure satisfies an $A_\infty$ type condition on a cube $Q_0$ then a stopping time argument allows us to extract a pairwise disjoint family $\F\subset \dd_{Q_0}$ such that the averages of the measure for cubes  ``above''  the sawtooth (i.e., in  $\dd_{\F,Q_0}$) are essentially constant. Additionally, the complement of the union of the cubes in $\F$ are an ample portion of $Q_0$, this means that the local sawtooth region $\Omega_{\F,Q_0}$ has an ample contact with $Q_0$.

\begin{lemma}\label{lemma:stop-time}
Let $Q_0\in\dd$ and let $\mu$ be a non-negative regular Borel measure on $Q_0$. Assume that
$\mu\ll \sigma$ on $Q_0$ and write $k=d\mu/d\sigma$. Assume also that there exist $K_0\ge 1$, $\theta>0$ such that
\begin{equation}
1\le\frac{\mu(Q_{0})}{\sigma(Q_{0})}\le K_0\qquad \mbox{ and }
\qquad
\frac{\mu(F)}{\sigma(Q_{0})}
\le K_0\,\left(\frac{\sigma(F)}{\sigma(Q_{0})}\right)^{\theta},
\quad
\forall\,F\subset Q_0.
\label{eq:mu-Ainf-norm}
\end{equation}
Then, there exists a pairwise disjoint family $\F=\{Q_j\}_j\subset\dd_{Q_0}\setminus \{Q_0\}$ such that
\begin{equation}
\sigma\Big(Q_0\setminus\bigcup_{Q_j\in \F}Q_j\Big)\ge K_1^{-1} \sigma(Q_0)
\label{eq:ample-saw}
\end{equation}
and
\begin{equation}
\frac12\le \frac{\mu(Q)}{\sigma(Q)}\le K_0\,K_1,
\qquad
\forall\,Q\in\dd_{\F,Q_0},
\label{eq:saw-mu}
\end{equation}
where $K_1:=(4\,K_0)^{\frac1{\theta}}$
\end{lemma}

\begin{proof}
The proof is based on a stopping time argument similar to those used in the
proof of the Kato square root conjecture \cite{HMc, HLMc, AHLMcT}, a more refined version appears in \cite{HLMN, HM4}.

Let $\F=\{Q_j\}_j$ be the collection of dyadic cubes contained in $Q_0$
that are maximal, and therefore pairwise disjoint,  with respect to the property that either
\begin{equation}\label{eqn:stop}
 \frac{\mu(Q)}{\sigma(Q)}<  \frac12
\qquad
\mbox{ or }
\qquad
 \frac{\mu(Q)}{\sigma(Q)} >K_0\,K_1.
\end{equation}
Note that \eqref{eq:mu-Ainf-norm} and the fact that $K_1>1$ imply that $\F\subset\dd_{Q_0}\setminus \{Q_{0}\}$. Also, the maximality of the cubes in $\F$ immediately gives \eqref{eq:saw-mu}.

On the other hand, we observe that $\F=\F_{1}\cup \F_{2}$ where $\F_{1}$ corresponds to the family stopping time cubes with respect to the first criterion in \eqref{eqn:stop} and $\F_{2}=\F\setminus\F_{1}$ is comprised of the maximal cubes for which the first condition in \eqref{eqn:stop} fails but the second holds. Set
$$
F_0=Q_0\setminus\bigg(\bigcup_{Q_j\in\F} Q_j\bigg),
\qquad\quad
F_1=\bigcup_{Q_j\in\F_1} Q_j,
\qquad\quad
F_2=\bigcup_{Q_j\in\F_2} Q_j,
$$
so that $Q_0=F_0\cup F_1\cup F_2$.

We first handle the cubes in $\F_1$ which, by construction, satisfy
$$
\mu(F_1)
=
\sum_{Q_j\in \F_1} \mu(Q_j)
\le
\frac12 \sum_{Q\in\F_{1}} \sigma(Q)
=
\frac12 \sigma(F_1)
\le
\frac12 \sigma(Q_0).
$$
On the other hand, the definition of the family $\F_2$ and \eqref{eq:mu-Ainf-norm} give
$$
\sigma(F_2)
=
\sum_{Q_j\in \F_{2}} \sigma(Q_j)
\le
\frac1{K_0\,K_1}\,\sum_{Q_j\in \F_{2}} \mu(Q_j)
=
\frac1{K_0\,K_1}\, \mu(F_2)
\le
\frac1{K_0\,K_1}\, \mu(Q_0)
\le
\frac{1}{K_1}\,\sigma(Q_0).
$$
This, \eqref{eq:mu-Ainf-norm}, and our choice of $K_1$ yield
$$
\frac{\mu(F_2)}{\sigma(Q_0)}
\le
K_0
\left(
\frac{\sigma(F_2)}{\sigma(Q_0)}
\right)^{\theta}
\le
\frac{K_0}{K_1^{\theta}}
=
\frac14.
$$
Collecting the estimates obtained for $F_1$ and $F_2$, and using again \eqref{eq:mu-Ainf-norm}
we see that
$$
\sigma(Q_0)
\le
\mu(Q_0)
=
\mu(F_0)+\mu(F_1)+\mu(F_2)
\le
\mu(F_0)+
\frac34\,\sigma(Q_0).
$$
Hiding the last term on the right hand side and by \eqref{eq:mu-Ainf-norm} one can conclude that
$$
\frac14
\le
\frac{\mu(F_0)}{\sigma(Q_0)}
\le
K_0
\left(\frac{\sigma(F_0)}{\sigma(Q_0)}\right)^{\theta},
$$
which is \eqref{eq:ample-saw} and the proof is complete.
\end{proof}

With a slight abuse of notation,  let $Q^0$ be either $\pom$, and in that case $\dd_{Q^0}:=\dd$, or a fixed cube in $\dd$, hence $\dd_{Q^0}$ is the family of dyadic subcubes of $Q^0$. Let $\alpha=\{\alpha_Q\}_{Q\in\dd_{Q^0}}$ be a sequence of non-negative numbers indexed by the dyadic ``cubes'' in $\dd_{Q^0}$, and for any collection $\dd'\subset\dd_{Q^0}$, we define an associated discrete ``measure''
\begin{equation}
\mut_\alpha(\dd'):= \sum_{Q\in\dd'}\alpha_{Q}.
\label{eq:mut-defi}
\end{equation}
We say that $\mut_\alpha$ is a ``Carleson measure'' (with respect to $\sigma$) in $Q^0$,  if
\begin{equation}\label{eq6.0}
\|\mut_\alpha\|_{\C(Q^0)}:= \sup_{Q\in\dd_{Q^0}} \frac{\mut_\alpha(\dd_{Q})}{\sigma(Q)} <\infty.
\end{equation}
For simplicity, when $Q^0=\pom$ we simply write $\|\mut_\alpha\|_{\C}$.

Our next result establishes that to show that $\mut_\alpha$ is a Carleson measure it suffices to check \eqref{eq6.0} only on ``sawtooths with an ample contact'':

\begin{lemma}\label{lemma:HN-Car}
Let $Q^0$ be either $\pom$ or a fixed cube in $\dd$. Let $\alpha=\{\alpha_Q\}_{Q\in\dd_{Q^0}}$ be a sequence of non-negative numbers and consider $\mut_\alpha$ as defined above. Given $M_1>0$ and $K_1\ge 1$, we assume that for every $Q_0\in\dd_{Q^0}$ there exists a pairwise disjoint family $\F_{Q_0}=\{Q_j\}_{j}\subset \dd_{Q_0}\setminus \{Q_0\}$ such that
\begin{equation}
\sigma\Big(Q_0\setminus\bigcup_{Q_j\in \F_{Q_0}}Q_j\Big)\ge K_1^{-1} \sigma(Q_0)
\label{eq:ample-sawtooth:lemma}
\end{equation}
and
\begin{equation}\label{eq:Car:sawtooth:lemma}
\mut_\alpha(\dd_{\F_{Q_0},Q_0})\le M_1\, \sigma(Q_0).
\end{equation}
Then, $\mut_\alpha$ is Carleson measure in $Q^0$ and moreover
\begin{equation}
\|\mut_\alpha\|_{\C(Q^0)}= \sup_{Q\in\dd_{Q^0}} \frac{\mut_\alpha(\dd_{Q})}{\sigma(Q)}
\le K_1\, M_1.
\label{eq:Carleson:lemma}
\end{equation}
\end{lemma}

\begin{proof}
We first take a sequence $\dd^1\subset\dd^2\subset\cdots \subset\dd^N\subset\cdots\subset\dd_{Q^0}$ such that $\dd_{Q^0}=\cup_N\dd^N$ and $\# \dd^N=N$. For each $N\ge 1$ we let $\alpha_N:=\{\alpha_Q^N\}_{Q\in \dd_{Q^0}}$ where
$\alpha_Q^N:=\alpha_Q$ if $Q\in\dd^N$ and $\alpha_Q^N:=0$ otherwise. Let $\mut_{\alpha_N}$ be the corresponding discrete measure associated with $\alpha_N$ and set $\ell_N=\min\{\ell(Q):Q\in \dd^N\}>0$.

We first note that
$$
\|\mut_{\alpha_N}\|_{\C(Q^0)}
=
\sup_{Q\in\dd_{Q^0}} \frac{\mut_{\alpha_N}(\dd_{Q})}{\sigma(Q)}
=
\sup_{Q\in\dd_{Q^0}, \ell(Q)\ge  \ell_N} \frac{1}{\sigma(Q)}\sum_{Q'\in\dd_{Q}\cap \dd^N} \alpha_{Q'}
\lesssim
\frac{1}{(\ell_N)^n}\,\sum_{Q'\in\dd^N} \alpha_{Q'} <\infty.
$$
Fix now $Q_0\in\dd_{Q^0}$ and let $\F_{Q_0}$ be the associated pairwise disjoint family given by our hypotheses. By the definition of $\dd_{\F_{Q_0},Q_0}$ and by \eqref{eq:ample-sawtooth:lemma} we have
\begin{multline*}
\mut_{\alpha_N} (\dd_{Q_0}\setminus\dd_{\F_{Q_0},Q_0})
=
\sum_{Q_j\in\F_{Q_0}} \mut_{\alpha_N}(\dd_{Q_j})
\le
\|\mut_{\alpha_N}\|_{\C(Q^0)}\,\sum_{Q_j\in\F_{Q_0}} \sigma(Q_j)
\\
=
\|\mut_{\alpha_N}\|_{\C(Q^0)}\,\sigma\Big(\bigcup_{Q_j\in \F_{Q_0}} Q_j\Big)
\le
(1-K_1^{-1})\,\|\mut_{\alpha_N}\|_{\C(Q^0)}\,\sigma(Q_0).
\end{multline*}
This and \eqref{eq:Car:sawtooth:lemma} yield
$$
\frac{\mut_{\alpha_N} (\dd_{Q_0})}{\sigma(Q_0)}
=
\frac{\mut_{\alpha_N} (\dd_{\F_{Q_0},Q_0})}{\sigma(Q_0)}
+
\frac{\mut_{\alpha_N} (\dd_{Q_0}\setminus\dd_{\F_{Q_0},Q_0})}{\sigma(Q_0)}
\le
M_1
+
(1-K_1^{-1})\,\|\mut_{\alpha_N}\|_{\C(Q^0)}.
$$
Note that this estimate holds for every $Q_0\in\dd_{Q^0}$. Hence, we conclude that
$$
\|\mut_{\alpha_N}\|_{\C(Q^0)}
=
\sup_{Q\in\dd_{Q^0}} \frac{\mut_{\alpha_N}(\dd_{Q})}{\sigma(Q)}
\le
M_1
+
(1-K_1^{-1})\,\|\mut_{\alpha_N}\|_{\C(Q^0)}.
$$
We can then hide the last term (which is finite as observed above) to obtain  $\|\mut_{\alpha_N}\|_{\C(Q^0)}\le K_1 \,M_1$ and letting $N\to\infty$ we conclude  \eqref{eq:Carleson:lemma}.
\end{proof}

\section{Proof of the main result}\label{section:proof-main}
Given $0<c_0<1$, let $\B=\B(c_0)$
denote the collection
of $Q\in\dd$ for which the $c_0$-exterior Corkscrew condition (see
Definition \ref{def1.dyadcork}) fails.
Set $\alpha:=\{\alpha_Q\}_{Q\in \dd}$ with
\begin{equation}\label{eq2.5}
\alpha_Q := \left\{
\begin{array}{ll}
\sigma(Q),&
\ {\rm if \ } Q \in \B,
\\[6pt]
0,&\
{\rm otherwise}.
\end{array}
\right.
\end{equation}
Associate to $\alpha$  the discrete measure $\mut_\alpha$ as above, which depends on the parameter $c_0$.
We are going to prove that under the assumptions in Theorem \ref{theor:main} the collection $\B$ satisfies a packing condition, i.e., that $\mut_\alpha$ is a discrete Carleson measure, provided that $c_0$ is small enough.
\begin{proposition}\label{prop:mu-Carleson}
Under the assumptions of Theorem \ref{theor:main}, there is $c_0$ sufficiently small and $M_1\ge 1$, such that if $\B=\B(c_0)$,
its associated measure $\mut_\alpha$ as above satisfies the packing condition
\begin{equation}\label{desired:Car}
\|\mut_\alpha\|_{\C}:=
\sup_{Q\in\dd} \,\frac{ \mut_\alpha(\dd_Q)}{\sigma(Q)}\,  =\, \sup_{Q\in\dd}\,\frac1{\sigma(Q)}
\sum_{Q'\in\B:\, Q'\subset Q} \sigma(Q') \,\leq M_1.
\end{equation}
The  constants $c_0$ and $M_1$  depend only upon dimension, $\Lambda$, the
1-sided CAD constants, $\big\||\nabla A|\,\delta\big\|_\infty$, $\|\nabla A\|_{\C(\Omega)}$ and finally $q$ and $C$ in \eqref{eq:higher-inte}.
\end{proposition}

Assuming this result momentarily, we fix a cube $Q\in\dd(\pom)$, and we seek to show that
$\Omega_{\rm ext}$ has a Corkscrew point relative to $Q$.
Let $\Delta_Q\subset Q$ denote the surface ball defined in \eqref{cube-ball}--\eqref{cube-ball2}.
Take $Q_1$, a sub-cube of $Q$ of maximal size contained in $\Delta_Q$, and observe that $\ell(Q_1)\geq c\ell(Q)$.
We claim that there exists $Q'\in \dd_{Q_1}\setminus \B$ such that $\ell(Q')\ge 2^{-[M_1]}\,\ell(Q_1)$ (here $[M_1]$ is the biggest integer smaller than or equal to $M_1$). Otherwise, by \eqref{desired:Car} (applied to $Q_1$)
$$
([M_1]+1)\,\sigma(Q_1)
=
\sum_{k=0}^{[M_1]} \sum_{\substack{Q\in\dd_{Q_1}\\ \ell(Q)=2^{-k}\,\ell(Q_1)}}\!\! \sigma(Q)
\le
\sum_{Q\in\B:\, Q\subset Q_1} \sigma(Q) \,\leq M_1\,\sigma(Q_1),
$$
which readily leads to a contradiction. Hence there is
$Q'\in \dd_{Q_1}\setminus \B$ such that $\ell(Q')\ge 2^{-[M_1]}\,\ell(Q_1)\ge  c\,2^{-[M_1]}\,\ell(Q)$. Since $Q'$ enjoys the $c_0$-exterior Corkscrew condition,  so does $Q$, but with $c_0$ replaced by $c_0'=c_0\,c\,2^{-[M_1]}$.
On the other hand, every surface ball contains a cube of comparable diameter, this means that there is an exterior Corkscrew point relative to every surface ball on the boundary, and therefore $\Omega$ is NTA, and hence chord-arc. This completes the proof of Theorem \ref{theor:main} modulo Proposition \ref{prop:mu-Carleson}.

\medskip

To prove Proposition \ref{prop:mu-Carleson} we are going to use Lemma \ref{lemma:HN-Car} with $Q^0=\pom$. Fix $Q_{0}\in\dd_{Q^0}=\dd$, an arbitrary dyadic cube, and our goal is to obtain \eqref{eq:Car:sawtooth:lemma} for some pairwise disjoint family $\F_{Q_0}\subset \dd_{Q_0}\setminus \{Q_0\}$ for which \eqref{eq:ample-sawtooth:lemma} holds. We note that it suffices to consider the case $\ell(Q_{0})<\diam(\pom)/M_0$ with $M_0$ large enough (depending only on the allowable parameters). In fact, assuming this, in order to prove the case $\diam(\pom)/M_0\le \ell(Q_0)\lesssim \diam(\pom)$ (of course this is meaningful only if $\diam(\pom)<\infty$), we cover $Q_0$ by disjoint cubes $\{Q_0^k\}_k$ with $\diam(\pom)/(2\,M_0)\le \ell(Q_0^k)<\diam(\pom)/M_0$. For each $Q_0^k$, by the previous case one can find $\F_{Q_0^k} $ so that \eqref{eq:ample-sawtooth:lemma} and \eqref{eq:Car:sawtooth:lemma} hold with $Q_0^k$ in place of $Q_0$. Note that if we set $\F_{Q_0}=\cup_k \F_{Q_0^k}$ we automatically have \eqref{eq:ample-sawtooth:lemma} and moreover
\begin{multline*}
\mut_\alpha(\dd_{\F_{Q_0},Q_0})
\le
\sum_{k}\mut_\alpha(\dd_{\F_{Q_0^k},Q_0^k})
+
\sum_{\substack{Q\in \dd_{Q_0}\cap\B
\\
\ell(Q)\ge \diam(\pom)/M_0
}} \sigma(Q)
\\
\le
M_1\,\sum_k \sigma(Q_0^k)
+
C_{M_0}\,\sigma(Q_0)
\le
(M_1+C_{M_0})\,\sigma(Q_0).
\end{multline*}
Thus, we have proved that for every $Q_0\in \dd$, \eqref{eq:Car:sawtooth:lemma} holds for some pairwise disjoint family $\F\subset \dd_{Q_0}\setminus \{Q_0\}$ satisfying \eqref{eq:ample-sawtooth:lemma}. Hence Lemma \ref{lemma:HN-Car} yields \eqref{desired:Car} with some constant $M_1'$ and hence the proof of Proposition \ref{prop:mu-Carleson} would be complete.

\medskip

In view of the previous observation we fix $Q_0\in \dd$ such that $\ell(Q_{0})<\diam(\pom)/M_0$. We choose $M_0$ so that if we set $X_{0}=X_{\sqrt{M_0}\Delta_{Q_{0}}}$ one has that $2\,\kappa_0\,r_{Q_{0}}\le \delta(X_{0})\le\sqrt{M_0}\,r_{Q_{0}}$, where we recall that $\kappa_0$ was chosen (depending only on the allowable parameters) so that \eqref{tent-Q} holds. In such a case, $\dist(X_0, T_{Q_0}^{**})\ge \kappa_0\,r_{Q_{0}}$, hence the pole $X_0$ will be away from where the argument takes place. By Lemma \ref{Bourgainhm} and Harnack's inequality there is $C_0\ge 1$ depending on the allowable parameters and $M_0$ such that $\omega_L^{X_{0}}(Q_{0})\ge C_0^{-1}$. We now normalize the elliptic measure and the Green function as follows
\begin{equation}
\omega:=C_0\,\sigma(Q_{0})\,\omega_L^{X_{0}}
\qquad\mbox{ and }
\qquad
\G(\cdot):=C_0\,\sigma(Q_{0})\,G_L(X_0,\,\cdot\,).
\label{eq:normalization}
\end{equation}
Note that away from $X_0$, $L^\top \G(\cdot)=C_0\,\sigma(Q_{0})\,L^\top G_L(X_0,\,\cdot\,)=C_0\,\sigma(Q_{0})\, L^\top G_{L^\top}(\cdot, X_0)=0$ 
(see Lemma \ref{lemma2.green}).
Moreover
by our choice of $X_0$, Lemmas \ref{lemma2.cfms} and \ref{lemma.double}, \eqref{cube-ball}, and \eqref{cube-ball2} it follows that
\begin{equation}
\frac{\G(X_Q)}{\ell(Q)}
\approx
\frac{\omega(Q)}{\sigma(Q)},
\qquad \forall\,Q\in\dd_{Q_0}
.
\label{eq:CFMS-above:ns}
\end{equation}
On the other hand, since $\omega_L^{X_0}(\pom)\le 1$,
\begin{equation}\label{hm-sgima-top}
1\le\frac{\omega(Q_{0})}{\sigma(Q_{0})}\le C_0.
\end{equation}
By assumption, $\omega\ll\sigma$, and if $k=d\omega/d\sigma$ denotes the normalized Poisson kernel it follows that, for $M_0$ is large enough, \eqref{eq:higher-inte}, \eqref{cube-ball}, and \eqref{cube-ball2} yield
$$
\left(\aver{Q_0} k(y)^q\,d\sigma(y)\right)^{\frac1q}\le C^{\frac1q}\,C_0=:K_0,
$$
where $C$ is the constant in \eqref{eq:higher-inte}. As a consequence of that, \eqref{hm-sgima-top} and  H\"older's inequality one can derive
\begin{equation}
\frac{\omega(F)}{\sigma(Q_{0})}
=
\aver{Q_0} 1_F(y)\,k(y)\,d\sigma(y)
\le
K_0\,\left(\frac{\sigma(F)}{\sigma(Q_{0})}\right)^{\frac1{q'}},
\qquad\forall\,F\subset Q_{0}.
\label{eq:omega-sigma-Ainfty-Q0}
\end{equation}
Hence we can apply Lemma \ref{lemma:stop-time} to $\mu=\omega$ and obtain a pairwise disjoint family $\F_{Q_0}=\{Q_j\}_j\subset\dd_{Q_0}\setminus \{Q_0\}$ verifying \eqref{eq:ample-saw} and \eqref{eq:saw-mu}. Thus, as observed before
(see Lemma \ref{lemma:HN-Car}) we wish to find $M_1$ independent of $Q_0$ such that
\begin{equation}\label{reduction-sawtooth}
\mut_\alpha(\dd_{\F_{Q_0},Q_0})\le M_1\, \sigma(Q_0).
\end{equation}
Hence, in what follows, $Q_0\in\dd$ and $\F_{Q_0}$ is a pairwise disjoint family $\F_{Q_0}=\{Q_j\}_j\subset\dd_{Q_0}\setminus \{Q_0\}$ verifying \eqref{eq:ample-saw} and
\begin{equation}
\frac12\le \frac{\omega(Q)}{\sigma(Q)}\le K_0\,K_1,
\qquad
\forall\,Q\in\dd_{\F_{Q_0},Q_0},
\label{eq:hm-sigma}
\end{equation}
which is \eqref{eq:saw-mu} with $\mu=\omega$ (see \eqref{eq:normalization}).

Let us now fix $Q\in\dd_{\F_{Q_0},Q_0}\cap \B$ and a point $z_Q\in \Delta_Q\subset Q$. Set $B:=B(z_Q,r/4)$,
with $r:=r_Q\approx \ell(Q)$, and $\Delta:=B\cap \pom$.
Take $\Phi\in C^\infty_0(B)$, with $0\leq\Phi\leq 1$, $\Phi\equiv 1$ on $\frac12B$, and $\|\nabla\Phi\|_\infty \lesssim r^{-1}$.
Let $0<\epsilon\ll 1$ to be chosen and set $\vec{\beta}:= |U_{Q,\epsilon}|^{-1}\iint_{U_{Q,\epsilon}} A^\top(Y)\, \nabla \G(Y)\,dY$. Recall that (see Section \ref{ss:grid}) $U_{Q,\epsilon}=\Omega_{\F_0(\epsilon \,\ell(Q)),Q}$ where $\F_0=\emptyset$ and hence $\F_0(\epsilon \,\ell(Q))$ is the collection of $Q'\in \dd$ such that $\epsilon\,\ell(Q)/2<\ell(Q')\le \epsilon\,\ell(Q)$. In particular, $Q\in \dd_{\F_0(\epsilon \,\ell(Q)),Q}$ and  \eqref{tent-Q} yields $
\interior(U_{Q})
\subset U_{Q,\epsilon}\subset
T_Q
\subset B_Q^*$. Notice that $\W_Q^*\neq\emptyset$ and hence there is $I\in\W_Q^*$ such that $I\subset \interior(U_{Q})$ with $\ell(Q)\approx \ell(I)$
and consequently $|U_{Q,\epsilon}|\approx \ell(Q)^{n+1}$. Keeping in mind the normalization \eqref{eq:normalization}, our choice of $X_0$ and \eqref{eq2.green4}, we have that $L^\top \G=0$ in the weak sense in $T_{Q_0}^{**}$. Thus,
 Caccioppoli's inequality, Harnack's inequality and \eqref{eq:CFMS-above:ns} yield that for every $I\in \W_{Q'}^*$, $Q'\in\dd_Q\subset \dd_{Q_0}$
\begin{align}\label{est:Caccc+CFMS}
\iint_{I^*} |\nabla \G(Y)|\,dY
\lesssim
|I| \frac{\G(X(I))}{\delta(X(I))}
\approx
\ell(Q')^n\,\G(X_{Q'})
\approx
\ell(Q')\,\omega(Q'),
\end{align}
and hence
\begin{multline}\label{est:beta}
|\vec{\beta}\,|
\lesssim
\ell(Q)^{-(n+1)}\iint_{T_Q} |\nabla \G(Y)|\,dY
\lesssim
\ell(Q)^{-(n+1)}\sum_{Q'\in\dd_Q} \sum_{I\in \W_{Q'}^*} \iint_{I^*} |\nabla \G(Y)|\,dY
\\
\lesssim
\ell(Q)^{-n}\sum_{k=0}^\infty 2^{-k}\sum_{\substack{Q'\in\dd_Q\\ \ell(Q')=2^{-k}\,\ell(Q)}} \omega(Q')
\lesssim
\frac{\omega(Q)}{\sigma(Q)}
\lesssim
1
,
\end{multline}
where we have used that $\W_{Q'}^*$ has uniformly bounded cardinality and the last estimate follows from \eqref{eq:hm-sigma} since $Q\in\dd_{\F_{Q_0},Q_0}$.

We next use Lemma \ref{lemma.double},  \eqref{eq2.14}, and \eqref{eq2.green4} (keeping in mind \eqref{eq:normalization}, \eqref{eq:hm-sigma} and moving slightly the pole $X_0$ if needed)
\begin{align}\label{omega-Q}
\sigma(Q)
&\approx
\omega(Q)
\approx
\int_{\pom}\Phi\,d\omega
=
-\iint_{\Omega} A^\top(X)\,\nabla \G(X)\cdot \nabla \Phi(X)\,dX
\\ \nonumber
&=
-\iint_{\Omega} \big(A^\top(X)\,\nabla \G(X)\,-\vec{\beta}\,\big)\cdot\nabla \Phi(X)\,dX
-
\iint_{\ree} \vec{\beta}\cdot\nabla \Phi(X)\,dX
+
\iint_{\Omega_{\rm ext}} \vec{\beta}\cdot\nabla \Phi(X)\,dX
\\ \nonumber
&=
-\iint_{\Omega} \big(A^\top(X)\,\nabla \G(X)\,-\vec{\beta}\,\big)\cdot\nabla \Phi(X)\,dX
+
\iint_{\Omega_{\rm ext}} \vec{\beta}\cdot\nabla \Phi(X)\,dX
\\ \nonumber
&=:-\I+\I\I.
\end{align}

We first estimate $\I\I$. By \cite[Lemma 5.7]{HM-URHM}, the failure of the $c_0$-exterior Corkscrew property
implies that $|\Omega_{\rm ext}\cap B|\lesssim c_0\,r^{n+1}$. This and  \eqref{est:beta} give
\begin{equation}\label{est-II}
|\I\I|
\lesssim
|\vec{\beta}\,|\,r^{-1}\,|\Omega_{\rm ext}\cap B|
\lesssim
\,c_0\, r^{n}
\approx
c_0\,\sigma(Q).
\end{equation}
To estimate $\I$ we proceed as follows.
\begin{align}\label{est-I}
|\I|
&\lesssim
r^{-1}\iint_{\Omega\cap B} \big|A^\top(X)\,\nabla \G(X)\,-\vec{\beta}\,\big|\,dX
\\ \nonumber
&\le
r^{-1}\Big( \iint_{U_{Q,\epsilon}}  \big|A^\top(X)\,\nabla \G(X)\,-\vec{\beta}\,\big|\,dX + \iint_{(\Omega\setminus U_{Q,\epsilon})\cap B}  \big|A^\top(X)\,\nabla \G(X)\,-\vec{\beta}\,\big|\,dX\Big)
\\ \nonumber
&=:
r^{-1}\,(\I_1+\I_2).\nonumber
\end{align}
For $\I_1$ we use H\"older's inequality, our choice of $\vec\beta$, Lemma \ref{poincare} and the fact that $\delta(X)\approx_\epsilon\ell(Q)$ for every $X\in U_{Q,\epsilon}$:
\begin{multline}\label{est-I1}
\I_1
\lesssim
\ell(Q)^{\frac{n+1}2}\,\Big( \iint_{U_{Q,\epsilon}} \big|A^\top(X)\,\nabla \G(X)\,-\vec{\beta}\,\big|^2\,dX \Big)^{\frac12}
\le
C_\epsilon\,
\ell(Q)^{\frac{n+3}2}\,\Big( \iint_{U_{Q,\epsilon}} \big|\nabla (A^\top\,\nabla \G)(X)|^2\,dX \Big)^{\frac12}
\\
\le
C_\epsilon\,
r\,\sigma(Q)^{\frac12}\,\Big( \iint_{U_{Q,\epsilon}} \big|\nabla (A^\top\,\nabla \G)(X)|^2\,\delta(X)\,dX \Big)^{\frac12}
=:
C_\epsilon\,
r\,\sigma(Q)^{\frac12}\,\Upsilon_{Q,\epsilon}^{\frac12}.
\end{multline}

Before estimating $\I_2$ we need to make the following observation. Let $I\in \W$ be such that
$I^*\cap B\neq\emptyset$ and pick $Y\in I^*\cap B$. In particular,
$$
4\diam(I)\le\dist(I,\pom)\le|Y-z_Q|<\frac{r}{4}\le\frac{\ell(Q)}{4}< \frac{\diam(\pom)}{4 M_0}.
$$
Recall the construction of $Q_I^*$, the unique dyadic cube satisfying that $z\in Q_I^*$ and $\ell(Q_I^*)=\ell(I)$, where $z\in \pom$ is such that
$\dist(I,\pom)=\dist(I,z)$. We claim that $Q_I^*\subset Q$. To show this let us take $Z\in I$ such that $\dist(I,\pom)=|Z-z|$. Then
$$
|z-x_Q|
\le
|z-Z|+|Z-Y|+|Y-z_Q|+|z_Q-x_Q|
\le
\dist(I,\pom)+\diam(I^*)+\frac{r}{4}+r
<2r.
$$
This implies that $z\in \Delta(x_Q,2r)\subset Q$ (cf. \eqref{cube-ball}) and since $\ell(Q_I^*) =\ell(I)<\ell(Q)$ it follows that $Q_I^*\subset Q$ by the dyadic properties.

We are now ready to estimate $\I_2$. We first see that by choice of $B$ we have that $B\cap \Omega\subset T_Q$. Indeed, let $Y\in B\cap \Omega$ and take $I\in \W$ so that $Y\in I$. Note that by the previous observation $Q_I^*\subset Q$.   Note also that, as mentioned above, our construction guarantees that $I\in\W_{Q_I^*}^*$. All these yield $Y\in I\subset\interior(I^*)\subset \interior(U_{Q_I^*})\subset T_Q$.

Once we have shown that  $B\cap \Omega\subset T_Q$ one can easily see that $(\Omega\setminus U_{Q,\epsilon})\cap B\subset T_Q\setminus U_{Q,\epsilon}\subset \Sigma_\epsilon:=\big\{X\in\Omega: \delta(X)\lesssim\epsilon\,\ell(Q)\big\}$. Therefore, if $\epsilon$ is small enough,
\begin{align}\label{est-I2}
\I_2
\lesssim
|\vec{\beta}\,|\,|B\cap\Sigma_\epsilon|
+
\iint_{B\cap\Sigma_\epsilon}|\nabla \G|\,dX
\lesssim
\epsilon\,\ell(Q) r^n+\I_3
\approx
r\,\epsilon\,\sigma(Q)+\I_3,
\end{align}
where we have used \eqref{est:beta} and \cite[Lemma 5.3]{HM-URHM}. To estimate $\I_3$, we use again the cubes $Q_I^*$ as above associated with $I\in \W$ with $I^*\cap B\neq\emptyset$. As already mentioned in such a scenario, $Q_I^*\subset Q$ and $I\in\W_{Q_I^*}^*$. We can then invoke \eqref{est:Caccc+CFMS} to obtain
\begin{multline*}
\I_3
\le
\sum_{\substack{I\in \W: I^*\cap B\neq\tinyemptyset\\ \ell(I)\lesssim \epsilon\,\ell(Q)}}
\iint_{I^*} |\nabla \G(X)|\,dX
\lesssim
\sum_{\substack{I\in \W: I^*\cap B\neq\tinyemptyset\\ \ell(I)\lesssim \epsilon\,\ell(Q)}}  \omega(Q_I^*)\, \ell(Q_I^*)
=
\sum_{k: 2^{-k}\lesssim \epsilon\,\ell(Q)}2^{-k}
\sum_{\substack{I\in \W: I^*\cap B\neq\tinyemptyset\\ \ell(I)=2^{-k}}}  \omega(Q_I^*).
\end{multline*}
Notice that for $k$ fixed, the family $\{Q_I^*\}_{I\in\W: \ell(I)=2^{-k}}$ has bounded overlap, hence
\begin{align}\label{est-I3}
\I_3
\lesssim
\omega(Q)\sum_{k: 2^{-k}\lesssim \epsilon\,\ell(Q)}2^{-k}
\lesssim
\epsilon\,\ell(Q)\,\omega(Q)
\lesssim
\epsilon\,r\,\sigma(Q),
\end{align}
where the last estimate follows again from \eqref{eq:hm-sigma} since $Q\in\dd_{\F_{Q_0},Q_0}$.
Plugging \eqref{est-II}, \eqref{est-I}, \eqref{est-I1}, \eqref{est-I2}, and \eqref{est-I3} into \eqref{omega-Q}
we conclude that
$$
\sigma(Q)
\lesssim
c_0\,\sigma(Q)
+C_\epsilon\,
\sigma(Q)^{\frac12}\,\Upsilon_{Q,\epsilon}^{\frac12}+
\epsilon\,\sigma(Q).
$$
If $\epsilon$ and $c_0$ are taken small enough (we may assume for later use that $\epsilon=2^{-N_\epsilon}$ for some $N_\epsilon\in \N$ large enough) the first and third term in the right hand side can be hidden and one easily arrives at
\begin{multline}
\sigma(Q)
\lesssim_{\epsilon,c_0}
\Upsilon_{Q,\epsilon}
=
\iint_{U_{Q,\epsilon}} \big|\nabla (A^\top\,\nabla \G)(X)|^2\,\delta(X)\,dX
\\
\le
\sum_{\substack{Q'\in\dd_Q \\ \epsilon\,\ell(Q)<\ell(Q')\le\ell(Q)}} \iint_{U_{Q'}} \big|\nabla (A^\top\,\nabla \G)(X)|^2\,\delta(X)\,dX
=:\sum_{\substack{Q'\in\dd_Q \\ \epsilon\,\ell(Q)<\ell(Q')\le\ell(Q)}}  \Upsilon_{Q'}.
\label{eq:main-sigma-Q:bad-saw}
\end{multline}
Hence, it has been shown that for a choice of $\epsilon$ and $c_0$ small enough, the previous estimate holds for all $Q\in\dd_{\F_{Q_0},Q_0}\cap \B(c_0)$. This in turn gives
\begin{equation}\label{reduction:sawtooth}
\mut_\alpha(\dd_{\F_{Q_0},Q_0})
=
\sum_{Q\in\dd_{\F_{Q_0},Q_0}\cap \B(c_0)} \sigma(Q)
\lesssim_{\epsilon,c_0}
\sum_{Q\in\dd_{\F_{Q_0},Q_0}}
\sum_{\substack{Q'\in\dd_Q \\ \epsilon\,\ell(Q)<\ell(Q')\le\ell(Q)}}  \Upsilon_{Q'}
\lesssim_\epsilon
\sum_{Q\in\dd_{\F_{Q_0}^\epsilon ,Q_0}}
\Upsilon_{Q},
\end{equation}
where $\F_{Q_0}^\epsilon$ is the pairwise disjoint family of $N_{\epsilon}$-descendants (recall that $\epsilon=2^{-N_{\epsilon}}$) of the elements of $\F_{Q_0}$:
$$
\F_{Q_0}^\epsilon
:=
\bigcup_{Q\in \F_{Q_0}} \big\{
Q'\in\dd_{Q}:  \ell(Q')=2^{-N_{\epsilon}}\,\ell(Q)=\epsilon \,\ell(Q)
\big\}.
$$
We next claim that
\begin{equation}\label{homogen-newsawtooth}
\frac{\omega(Q)}{\sigma(Q)} \approx_\epsilon 1,
\qquad
\forall\,Q\in \dd_{\F_{Q_0}^\epsilon,Q_0},
\end{equation}
that is, both estimates in \eqref{eq:hm-sigma} can be transmitted from $\dd_{\F_{Q_0},Q_0}$ to $\dd_{\F_{Q_0}^\epsilon,Q_0}$, albeit with bounds that may depend on $\epsilon$. To obtain that, fix $Q\in \dd_{\F_{Q_0}^\epsilon,Q_0}$. By \eqref{eq:hm-sigma}, we may assume that $Q\notin \dd_{\F_{Q_0},Q_0}$. This means that there is $Q_1\in \F_{Q_0}\subset\dd_{Q_0}$ such that $Q\subset Q_1$. Since $Q_1$ splits into its $N_\epsilon$-descendants, we can find $Q_1'\in\F_{Q_0}^\epsilon$ such that $Q_1'\cap Q\neq\emptyset$ and $\ell(Q_1')=\epsilon\,\ell(Q_1)$. In turn, since $Q\in\dd_{\F_{Q_0}^\epsilon,Q_0}$, it then follows that $Q_1'\subsetneq Q\subset Q_1$. We use this, the AR property, the doubling property of $\omega$ (Lemma \ref{lemma.double})  and \eqref{eq:hm-sigma}
(which holds for the dyadic parent of $Q_1$, denoted by $\widehat{Q}_1$, since $Q_1 \in \F_{Q_0}\subset\dd_{Q_0}\setminus\{Q_0\}$)
$$
\frac{\omega(Q)}{\sigma(Q)}
\approx_\epsilon
\frac{\omega(\widehat{Q}_1)}{\sigma(\widehat{Q}_1)}
\approx
1.
$$
This shows our claim.

Set $\widetilde{\alpha}:=\{\widetilde{\alpha}_Q\}_{Q\in \dd_{Q_0}}$ with
\begin{equation}\label{eq2.5:ns}
\widetilde{\alpha}_Q := \left\{
\begin{array}{ll}
\Upsilon_{Q},&
\ {\rm if \ } Q \in \dd_{\F_{Q_0}^\epsilon,Q_0},
\\[6pt]
0,&\
{\rm otherwise},
\end{array}
\right.
\end{equation}
and associate to $\widetilde{\alpha}$  the discrete measure $\mut_{\widetilde{\alpha}}$ as in \eqref{eq:mut-defi}. Then,  we may immediately see that \eqref{reduction-sawtooth} follows from \eqref{reduction:sawtooth} and
\begin{equation}
\mut_{\widetilde{\alpha}}(\dd_{Q_0})\lesssim_\epsilon\sigma(Q_0).
\label{eq:reduction:sawtooth:new}
\end{equation}

Our goal is then to obtain \eqref{eq:reduction:sawtooth:new} and in order to do that we shall distinguish between two cases depending whether or not $A$ is symmetric. The main idea is that when $A$ is symmetric, in the expression $\Upsilon_Q$ we can replace $\delta(X)$ by $\G(X)$ for every $X\in U_Q$ and for every $Q \in \dd_{\F_{Q_0}^\epsilon,Q_0}$. Doing this, 
\eqref{eq:reduction:sawtooth:new} will be obtained by an integration by parts argument. In the non-symmetric case, for the integration by parts to work, we would need to do the same thing but rather than $\G$ (which is essentially $G_L(X_0,\cdot)$, hence a null solution for $L^\top$, cf. \eqref{eq:normalization}), we would need to work with essentially $G_{L^\top}(X_0,\cdot)$, hence a null solution for $L$. The latter would require to perform the stopping time in Lemma \ref{lemma:stop-time} with $\omega_{L^\top}^{X_0}$. However, it is not clear that one can apply Lemma \ref{lemma:stop-time} simultaneously to $\omega_{L}^{X_0}$ and $\omega_{L^\top}^{X_0}$ and obtain a family of cubes whose complement is still ample. We are going to overcome this by another use of Lemma \ref{lemma:HN-Car}, hence we will work in $\widetilde{Q}_0\in\dd_{Q_0}$ and apply  Lemma \ref{lemma:stop-time} to $\omega_{L}^{\widetilde{X}_0}$ with $\widetilde{X}_0$ being effectively a corkscrew point relative to $\widetilde{Q}_0$.

\subsection{The symmetric case}\label{section-sym}
In this section we assume that $A$ is symmetric. We start observing that for every $Q\in \dd_{\F_{Q_0}^\epsilon,Q_0}$ and every $X\in U_Q$ we have that $\delta(X)\approx_\epsilon \G(X)$ for every $X\in U_Q$ in view of Harnack's inequality, \eqref{eq:CFMS-above:ns} and \eqref{homogen-newsawtooth}. This and the definitions of the sets $\{U_{Q}\}_{Q\in\dd_{Q_0}}$ and $\Omega^\ast_{\F^\epsilon_{Q_0},Q_0}$ (see Section \ref{ss:grid}) yield
\begin{multline}\label{eqn:new-goal-symm}
\mut_{\widetilde{\alpha}}(\dd_{Q_0})
=
\sum_{Q\in \dd_{\F_{Q_0}^\epsilon,Q_0}}\Upsilon_Q
\lesssim_\epsilon
\sum_{Q\in \dd_{\F_{Q_0}^\epsilon,Q_0}}
\iint_{U_Q} \big|\nabla (A\,\nabla \G)(X)|^2\,\G(X)\,dX
\\
\lesssim
\iint_{\Omega_{\F_{Q_0}^\epsilon,Q_0}^*} \big|\nabla (A\,\nabla \G)(X)|^2\,\G(X)\,dX.
\end{multline}
We next take an arbitrary $N$ large enough and define $\F_N:=\F_{Q_0}^\epsilon\big(2^{-N}\,\ell(Q_0)\big)$ as in
Section \ref{ss:grid}. That is, $\F_N\subset\dd_{Q_0}$ is the family of maximal cubes of the collection $\F_{Q_0}^\epsilon$ augmented by adding all dyadic cubes of size smaller than or equal than $2^{-N}\,\ell(Q_0)$. In particular, $Q\in\dd_{\F_N, Q_0}$ if and only if $Q\in \dd_{\F_{Q_0}^\epsilon,Q_0}$ and $\ell(Q)>2^{-N}\,\ell(Q_0)$. Clearly, $\dd_{\F_N, Q_0}\subset \dd_{\F_{N'}, Q_0}$ if $N\le N'$ and therefore $\Omega_{\F_{N},Q_0}^*\subset \Omega_{\F_{N'},Q_0}^*\subset \Omega_{\dd_{\F_{Q_0}^\epsilon},Q_0}^*$. This and the monotone convergence theorem give that
\begin{equation}
\iint_{\Omega_{\F_{Q_0}^\epsilon,Q_0}^*} \big|\nabla (A\,\nabla \G)(X)|^2\,\G(X)\,dX
=
\lim_{N\to\infty} \iint_{\Omega_{\F_{N},Q_0}^*} \big|\nabla (A\,\nabla \G)(X)|^2\,\G(X)\,dX.
\label{eq:Omega_N-TCM}
\end{equation}

We now formulate an auxiliary result that will easily lead us to the desired estimate.
We note that the
following proposition was previously announced, with a sketch of the proof, in \cite{ABHM}.  In the sequel, we shall
present the full details,
and treat also the non-symmetric case (see Proposition \ref{prop:CME-G:ns} below).

\begin{proposition}\label{prop:CME-G}
Assuming that $A$ is symmetric and given $C_1\ge 1$, one can find $C$ depending on $C_1$ and the allowable parameters such that if $\F_N\subset \dd_{Q_0}$, $N\ge 1$, is a family of pairwise disjoint dyadic cubes satisfying
\begin{equation}\label{goal:N:relevant}
C_1^{-1}
\le
\frac{\omega(Q)}{\sigma(Q)} \le C_1
\qquad\mbox{and}\qquad
\ell(Q)>2^{-N}\,\ell(Q_0),
\qquad \forall\,Q\in \dd_{\F_{N},Q_0},
\end{equation}
then
\begin{equation}\label{eqn-goal:N}
\iint_{\Omega_{\F_{N},Q_0}^*} \big|\nabla (A\,\nabla \G)(X)|^2\,\G(X)\,dX
\le
C \, \sigma(Q_0).
\end{equation}
\end{proposition}

Assuming this result momentarily we see that \eqref{homogen-newsawtooth} and the construction of $\F_N$
give \eqref{goal:N:relevant}. Thus  \eqref{eq:reduction:sawtooth:new} follows from \eqref{eqn:new-goal-symm}, \eqref{eq:Omega_N-TCM}, and \eqref{eqn-goal:N}. This completes the proof of Proposition \ref{prop:mu-Carleson} when $A$ is symmetric,  modulo obtaining the just stated proposition.

\begin{proof}[Proof of Proposition \ref{prop:CME-G} for the Laplacian]
The case when $L$ is the Laplacian (that is, $A$ is the identity matrix and hence $(a)$, $(b)$, and $(c)$ of Hypothesis \ref{hyp1} are trivial) is rather simple and models the general case.    To fix ideas, we first present this simple case.

Suppose for now that $L$ is the Laplacian.   We first observe that
\begin{equation}\label{est:G:saw-Lap}
\delta(X)|\nabla ^2\G(X)|,\
|\nabla \G(X)|\lesssim \frac{\G(X)}{\delta(X)}
\approx
1,
\qquad
\forall\, X\in \overline{\Omega_{\F_{N},Q_0}^{*}}
=:
\overline{\Omega_\star}
,
\end{equation}
albeit with bounds that depend on $C_1$ and the allowable parameters but which are uniform in $N$.
To see the ``$\lesssim$'' we use the harmonicity of $\G$, $\nabla \G$ and $\nabla^2 \G$; interior estimates, and Caccioppoli's inequality (we recall that we chose $X_0$ so that it is away from $T_{Q_0}^{*}$ and also that  $2^{-N}\,\ell(Q_0)\lesssim \delta(X)\lesssim \ell(Q_0)$ for every $X\in \overline{\Omega_\star}$). The proof ``$\approx$'' is as follows: given $X\in \overline{\Omega_\star} $, there is $Q\in \dd_{\F_{N},Q_0}$ and $I\in \W_Q^*$ such that $X\in I^{**}$.
Note that $\delta(X)\approx\delta(X_Q)\approx\ell(I)\approx\ell(Q)$ and $|X-X_Q|\lesssim \ell(Q)$. This, Harnack's inequality, \eqref{eq:CFMS-above:ns},  and \eqref{goal:N:relevant} yield as
desired
$$
\frac{\G(X)}{\delta(X)}
\approx
\frac{\G(X_{Q})}{\ell(Q)}
\approx
\frac{\omega(Q)}{\sigma(Q)}
\approx
1.
$$

We now proceed to obtain \eqref{eqn-goal:N} with $A$ being the identity matrix. Write ``$\partial$'' to denote a fixed generic derivative. We use that $\G$ and $\partial \G$ are harmonic in  $\overline{\Omega_\star}$
to see that
in that set the following pointwise equalities hold
\[
\div\nabla \big((\partial \G)^{2}\big)
=
2\,\div\big[(\partial \G)\,\nabla(\partial \G)\big]=2\,|\nabla(\partial \G)|^{2}
\]
 and
\[
\left[\div\nabla\big((\partial \G)^{2}\big)\right]\,\G
=
\div\left[\nabla\big((\partial\G)^{2}\big)\,\G\right] -\nabla\big((\partial\G)^{2}\big)\cdot\nabla \G
=
\div\left[\nabla \big((\partial\G)^{2}\big)\,\G -(\partial\G)^{2}\nabla \G\right].
\]
Note that $\Omega_\star$ is a finite union of fattened Whitney boxes, thus, its (outward) unit normal $\nu$ is well defined a.e.~on $\partial\Omega^\star$. Hence the divergence theorem can be applied to obtain
\begin{multline}
\label{intbypartsomegstar}
2\,\iint\limits_{\Omega^{\star}} |\nabla(\partial \G)(X)|^{2}\,\G(X)\,dX
=
\int\limits_{\partial\Omega^{\star}}\big(\nabla\big( (\partial\G)^{2}\big)\,\G-(\partial\G)^{2}\,\nabla \G\big)\cdot\nu\,
dH^n
\\
\lesssim
\int\limits_{\partial\Omega^{\star}} \big( |\nabla^2 \G |\,|\nabla \G|\,\G+|\nabla \G|^3\big)\,dH^n
\lesssim
H^n(\partial \Omega^{\star})
\lesssim
\ell(Q_0)^n
\approx
\sigma(Q_0),
\end{multline}
where we have used \eqref{est:G:saw-Lap}, that $\partial \Omega_{\star}$ is AR (cf. \cite[Lemma 3.61]{HM-URHM})
  and finally that $\diam(\pom_{\star})\approx\ell(Q_0)$ (note that all bounds are independent of $N$). From \eqref{intbypartsomegstar}, we immediately obtain  \eqref{eqn-goal:N} in the case of the Laplacian.
\end{proof}

Looking at the previous argument the matrix $A$ being non-constant (for
both the symmetric and non-symmetric cases) raises several issues. The first one appears in \eqref{est:G:saw-Lap}: the ``$\approx$'' is still correct but one does not expect to have the ``$\lesssim$'' for general matrices $A$ since, as opposed to the constant coefficient case, we no longer have that $\partial \G$ is a null solution of $L$. As we shall see below in Lemma \ref{eqn:point-nablaA:general}, under the assumption that $A$ satisfies $(b)$ of Hypothesis \ref{hyp1}, one can prove that
the estimate for $\nabla \G$ in \eqref{est:G:saw-Lap} holds pointwise and the estimate for $\nabla^2\G$
holds in a  $L^2$-average sense via a Caccioppoli type estimate for second derivatives of solutions. The second issue is that the presence of $A$ in \eqref{eqn-goal:N} makes the algebra significantly more difficult as one has to distribute derivatives and some of them hit $A$. Finally, because the estimates for $\nabla^2\G$ hold in an average sense, we cannot integrate by parts as in \eqref{intbypartsomegstar}. We will solve this by producing some wiggling after incorporating a smooth cut-off of the domain (see Lemma \ref{lemma:approx-saw}) which will have the effect of replacing integrals on the boundary by ``solid'' integrals in a ``strip'' along the boundary.

\begin{proof}[Proof of Proposition \ref{prop:CME-G}]
We just need to invoke Proposition \ref{prop:CME-G:ns} below with $Q_0=\widetilde{Q}_0$
since \eqref{goal:N:relevant:ns} follows at once from \eqref{goal:N:relevant}. Further details are left to the interested reader.
\end{proof}

\subsection{The non-symmetric case}\label{section-non-sym}
As explained above in the non-symmetric case we are going to need to use again Lemma \ref{lemma:HN-Car}. Recall that our goal is to show \eqref{eq:reduction:sawtooth:new}. Applying Lemma \ref{lemma:HN-Car} with $Q^0=Q_0$, it suffices to take an arbitrary $\widetilde{Q}_0\in\dd_{Q_0}$ and show that there exists
a pairwise disjoint family $\widetilde{\F}_{\widetilde{Q}_0}=\{\widetilde{Q}_j\}_{j}\subset \dd_{\widetilde{Q}_0}\setminus \{\widetilde{Q}_0\}$ such that
\begin{equation}\label{eq:Q_1-saw}
\sigma\Big(\widetilde{Q}_0\setminus\bigcup_{\widetilde{Q}_j\in \widetilde{\F}_{\widetilde{Q}_0}}\widetilde{Q}_j\Big)\ge \widetilde{K}_1^{-1} \sigma(\widetilde{Q}_0)
\qquad\mbox{and}\qquad
\mut_{\widetilde{\alpha}}(\dd_{\widetilde{\F}_{\widetilde{Q}_0},\widetilde{Q}_0})\le \widetilde{M}_1\sigma(\widetilde{Q}_0).
\end{equation}
In order to obtain this we note that $\ell(\widetilde{Q}_0)\le \ell(Q_0)<\diam(\pom)/M_0$ and set $\widetilde{X}_{0}=X_{\sqrt{M_0}\Delta_{\widetilde{Q}_0}}$. Recall that by choice of $M_0$ we have that $2\,\kappa_0\,r_{\widetilde{Q}_0}\le \delta(\widetilde{X}_0)\le\sqrt{M_0}\,r_{\widetilde{Q}_0}$, where $\kappa_0$ was chosen (depending only on the allowable parameters) so that \eqref{tent-Q} holds. In such a case, $\dist(\widetilde{X}_0, T_{\widetilde{Q}_0}^{**})\ge \kappa_0\,r_{\widetilde{Q}_0}$, hence the pole $\widetilde{X}_0$ will be away from where the argument takes place. By  applying  Lemma \ref{Bourgainhm} and Harnack's inequality we have $\omega_{L^\top}^{\widetilde{X}_0}(\widetilde{Q}_0)\ge \widetilde{C}_0^{-1}$ with $\widetilde{C}_0$ depending on the allowable parameters and $M_0$. We now take a normalization of the elliptic measure and the Green function for $L^\top$:
\begin{equation}
\omega_\top:=\widetilde{C}_0 \,\sigma(\widetilde{Q}_0)\,\omega_{L^\top}^{\widetilde{X}_0}
\qquad\mbox{ and }
\qquad
\G_\top(\cdot):=\widetilde{C}_0\,\sigma(\widetilde{Q}_0)\,G_{L^\top}( \widetilde{X}_0, \,\cdot\,, ).
\label{eq:normalization-ns:top}
\end{equation}
As before $L \G_\top=0$ away from $\widetilde{X}_0$, and by our choice of $\widetilde{X}_0$, Lemmas \ref{lemma2.cfms} and \ref{lemma.double}, \eqref{cube-ball}, and \eqref{cube-ball2} it follows that
\begin{equation}
\frac{\G_\top(X_Q)}{\ell(Q)}
\approx
\frac{\omega_\top(Q)}{\sigma(Q)},
\qquad \forall\,Q\in\dd_{\widetilde{Q}_0}
.
\label{eq:CFMS-above:ns:top}
\end{equation}

Since $\omega_{L^\top}^{\widetilde{X}_0}(\pom)\le 1$, it follows that $1\le\omega_\top(\widetilde{Q}_{0})/\sigma(\widetilde{Q}_{0})\le \widetilde{C}_0$. This and the fact that $\omega_{L^\top}$ (and hence $\omega_\top$) is in  $A_\infty(\pom)$  allow us to invoke much as before Lemma \ref{lemma:stop-time} with $\mu=\omega_\top$ to extract a family of pairwise disjoint cubes $\widetilde{\F}_{\widetilde{Q}_0}=\{\widetilde{Q}_j\}_{j}\subset \dd_{\widetilde{Q}_0}\setminus \{\widetilde{Q}_0\}$ such that
\begin{equation}\label{eq:Q_1-saw-stop}
\sigma\Big(\widetilde{Q}_0\setminus\bigcup_{\widetilde{Q}_j\in \widetilde{\F}_{\widetilde{Q}_0}}\widetilde{Q}_j\Big)\gtrsim \sigma(\widetilde{Q}_0)
\qquad\mbox{and}\qquad
\frac{\omega_\top(Q)}{\sigma(Q)}\approx 1,\quad\forall Q\in\dd_{\widetilde{\F}_{\widetilde{Q}_0},\widetilde{Q}_0},
\end{equation}
with implicit constants depending on $\widetilde{C}_0$ and the $A_\infty(\pom)$ character of $\omega_{L^\top}$. Consequently, in view of the previous considerations and Lemma \ref{lemma:HN-Car}, it remains to show  $\mut_{\widetilde{\alpha}}(\dd_{\widetilde{\F}_{\widetilde{Q}_0},\widetilde{Q}_0})\le \widetilde{M}_1\sigma(\widetilde{Q}_0)$.

Note first that if $\widetilde{Q}_0\not\in\dd_{\F_{Q_0}^\epsilon,Q_0}$ then $\widetilde{\alpha}_Q=0$ for every $Q\in\dd_{\widetilde{Q}_0}$, hence the desired estimate follows trivially since $\mut_{\widetilde{\alpha}}(\dd_{\widetilde{\F}_{\widetilde{Q}_0},\widetilde{Q}_0})=0$. Thus we may assume that $\widetilde{Q}_0\in\dd_{\F_{Q_0}^\epsilon,Q_0}$.
Write $\F_\star$ to denote the collection of maximal cubes in $(\F_{Q_0}^\epsilon\cap \dd_{\widetilde{Q}_0})\cup \widetilde{\F}_{\widetilde{Q}_0}$ so that
$\dd_{\F_{Q_0}^\epsilon,Q_0}\cap \dd_{\widetilde{\F}_{\widetilde{Q}_0},\widetilde{Q}_0}=\dd_{\F_\star,\widetilde{Q}_0}$. Hence,
\begin{multline}\label{reduction:sawtooth:ns}
\mut_{\widetilde{\alpha}}(\dd_{\widetilde{\F}_{\widetilde{Q}_0},\widetilde{Q}_0})
=
\sum_{Q\in \dd_{\widetilde{\F}_{\widetilde{Q}_0},\widetilde{Q}_0}}\widetilde{\alpha}_Q
=
\sum_{Q\in \dd_{\F_{Q_0}^\epsilon,Q_0}\cap \dd_{\widetilde{\F}_{\widetilde{Q}_0},\widetilde{Q}_0}} \Upsilon_{Q}
\\
=
\sum_{Q\in \dd_{\F_\star,\widetilde{Q}_0}} \iint_{U_{Q}} \big|\nabla (A^\top\,\nabla \G)(X)|^2\,\delta(X)\,dX
\lesssim
\iint_{\Omega_{\F_{\star},\widetilde{Q}_0}^*} \big|\nabla (A^\top\,\nabla \G)(X)|^2\,\G_\top(X)\,dX.
\end{multline}
Note that in the last estimate we have used that the sets $\{U_{Q}\}_{Q\in\dd_{\widetilde{Q}_0}}$ have bounded overlap and that Harnack's inequality,
\eqref{eq:CFMS-above:ns:top} and the second estimate in \eqref{eq:Q_1-saw-stop}  yield
$$
\frac{\G_\top(X)}{\delta(X)}
\approx
\frac{\G_\top(X_Q)}{\ell(Q)}
\approx
\frac{\omega_\top(Q)}{\sigma(Q)}\approx 1,
\qquad\forall\,X\in U_{Q},\quad
\forall\,Q\in\dd_{\widetilde{\F}_{\widetilde{Q}_0},\widetilde{Q}_0}.
$$
As in the symmetric case we take $N$ large enough and define $\F_N:=\F_{\star}\big(2^{-N}\,\ell(\widetilde{Q}_0)\big)$ as in
Section \ref{ss:grid} so that $Q\in\dd_{\F_N, \widetilde{Q}_0}$ if and only if $Q\in \dd_{\F_{\star},\widetilde{Q}_0}$ and $\ell(Q)>2^{-N}\,\ell(\widetilde{Q}_0)$. Clearly, $\dd_{\F_N, \widetilde{Q}_0}\subset \dd_{\F_{N'}, \widetilde{Q}_0}$ if $N\le N'$ and therefore $\Omega_{\F_{N},\widetilde{Q}_0}^*\subset \Omega_{\F_{N'},\widetilde{Q}_0}^*\subset \Omega_{\dd_{\F_{\star}},\widetilde{Q}_0}^*$. This and the monotone convergence theorem give that
\begin{equation}
\iint_{\Omega_{\F_{\star},\widetilde{Q}_0}^*} \big|\nabla (A^\top\,\nabla \G)(X)|^2\,\G_\top(X)\,dX
=
\lim_{N\to\infty} \iint_{\Omega_{\F_{N},\widetilde{Q}_0}^*} \big|\nabla (A^\top\,\nabla \G)(X)|^2\,\G_\top(X)\,dX.
\label{eq:Omega_N-TCM:ns}
\end{equation}
We can now state the analog of Proposition \ref{prop:CME-G} in this non-symmetric case.

\begin{proposition}\label{prop:CME-G:ns}
Given $C_1\ge 1$, one can find $C$ depending on $C_1$ and the allowable parameters such that if $\F_N\subset \dd_{\widetilde{Q}_0}$, $N\ge 1$, is a family of pairwise disjoint dyadic cubes satisfying
\begin{equation}\label{goal:N:relevant:ns}
C_1^{-1}
\le
\frac{\omega(Q)}{\sigma(Q)} \le C_1,
\qquad
C_1^{-1}
\le
\frac{\omega_\top(Q)}{\sigma(Q)} \le C_1,
\qquad\mbox{and}\qquad
\ell(Q)>2^{-N}\,\ell(\widetilde{Q}_0),
\end{equation}
for all $Q\in \dd_{\F_{N},\widetilde{Q}_0}$, then
\begin{equation}\label{eqn-goal:N:ns}
\iint_{\Omega_{\F_{N},\widetilde{Q}_0}^*} \big|\nabla (A^\top\,\nabla \G)(X)|^2\,\G_\top(X)\,dX
\le
C \, \sigma(\widetilde{Q}_0).
\end{equation}
\end{proposition}

Assuming this result momentarily, we note that  \eqref{goal:N:relevant:ns} follows from \eqref{homogen-newsawtooth}, the second item in \eqref{eq:Q_1-saw-stop}
and the construction of $\F_N$. Consequently,  \eqref{eqn-goal:N:ns}, \eqref{reduction:sawtooth:ns}, and \eqref{eq:Omega_N-TCM:ns} allow us to conclude that $\mut_{\widetilde{\alpha}}(\dd_{\widetilde{\F}_{\widetilde{Q}_0},\widetilde{Q}_0})\lesssim  \sigma(\widetilde{Q}_0)$. As observed above, this was the only thing left to obtain \eqref{eq:reduction:sawtooth:new} in the non-symmetric case and the proof of our main result is eventually complete.

Before starting the proof of Proposition \ref{prop:CME-G:ns} we need some auxiliary results, whose proofs are postponed until the next section.

\begin{lemma}\label{lemma:est-nabla:II}
Let $\Omega$ be an open set and let $A$ be a uniformly elliptic matrix in $\Omega$. Given $K\ge 0$, there exists $C_{K}$ depending only on ellipticity and $K$
such that  if
\begin{equation}\label{eqn:point-nablaA:general}
\sup_{X\in \Omega} |\nabla A(X)|\,\delta(X)\le K,
\end{equation}
then the following hold:
\begin{list}{$(\theenumi)$}{\usecounter{enumi}\leftmargin=.8cm
\labelwidth=.8cm\itemsep=0.2cm\topsep=.1cm
\renewcommand{\theenumi}{\roman{enumi}}}
\item For every $ u\in W^{1,2}_{\rm loc}(\Omega)$ , $u\ge 0$, verifying  $Lu=0$ in $\Omega$ in the weak-sense,
\begin{equation}\label{eqn:Caccioppoli:general}
|\nabla u(X)|\le C_K\,\frac{u(X)}{\delta(X)},
\qquad
\forall\,X\in\Omega.
\end{equation}

\item Given any cube $I\subset\ree$, if $6\,I\subset \Omega$ and $u\in W^{1,2}(6I)$
satisfies $Lu=0$ in $6I$ in the weak-sense,
\begin{equation}\label{eqn:Caccioppoli:nabla:general}
\iint_{I} |\nabla^2 u(Y)|^2\,dY
\le
\frac{C_K}{\ell(I)^2}\,\iint_{2\,I} |\nabla u(Y)|^2\,dY.
\end{equation}
\end{list}
\end{lemma}

\medskip

\begin{proof}[Proof of Proposition \ref{prop:CME-G:ns}]
Write $\Omega_\star=\Omega_{\F_{N},\widetilde{Q}_0}^{**}$. Let us first prove that
\begin{equation}\label{est:G:saw}
|\nabla \G(X)|\lesssim \frac{\G(X)}{\delta(X)}
\approx
1,\qquad
|\nabla \G_\top(X)|\lesssim \frac{\G_\top(X)}{\delta(X)}
\approx
1,
\qquad
\forall\, X\in \overline{\Omega_\star},
\end{equation}
albeit with bounds that depend on $C_1$ and the allowable parameters but which are uniform in $N$.
First, $|\nabla\G|\lesssim G/\delta$ (respectively $|\nabla\G_\top|\lesssim G_\top/\delta$) follows from Lemma \ref{lemma:est-nabla:II} applied to $u=\G$
(resp. $u=\G_\top$) where the implicit constant depends on ellipticity and $\big \||\nabla A|\,\delta\big\|_\infty$. To justify the use of that lemma we first notice that \eqref{eqn:point-nablaA:general} is just our assumption $(b)$ in Hypothesis \ref{hyp1}. Also, we recall we chose $X_0$ and $\widetilde{X}_0$
so that they are away from $T_{\widetilde{Q}_0}^{**}$ (indeed $X_0$ is away from $T_{Q_0}^{**}\supset T_{\widetilde{Q}_0}^{**}$). Hence $L\G_\top=0$ and $L^\top\G=0$ in the weak sense in $T_{\widetilde{Q}_0}^{**}$ (cf. Lemma \ref{lemma2.green}). Finally we observe that  $2^{-N}\,\ell(\widetilde{Q}_0)\lesssim \delta(X)\lesssim \ell(\widetilde{Q}_0)$ for every $X\in \overline{\Omega_\star}$

To continue with the proof of \eqref{est:G:saw} let $X\in \overline{\Omega_\star}$. Then there is $Q\in \dd_{\F_{N},\widetilde{Q}_0}$ and $I\in \W_Q^*$ such that $X\in I^{***}$.
Note that $\delta(X)\approx\delta(X_Q)\approx\ell(I)\approx\ell(Q)$ and $|X-X_Q|\lesssim \ell(Q)$. This, Harnack's inequality, \eqref{eq:CFMS-above:ns},  \eqref{eq:CFMS-above:ns:top}, that $\widetilde{Q}_0\in\dd_{Q_0}$ and \eqref{goal:N:relevant:ns} yield as
desired
$$
\frac{\G(X)}{\delta(X)}
\approx
\frac{\G(X_{Q})}{\ell(Q)}
\approx
\frac{\omega(Q)}{\sigma(Q)}
\approx
1,
\qquad
\frac{\G_\top(X)}{\delta(X)}
\approx
\frac{\G_\top(X_{Q})}{\ell(Q)}
\approx
\frac{\omega_\top(Q)}{\sigma(Q)}
\approx
1.
$$

We now proceed to obtain \eqref{eqn-goal:N:ns}. We note that by the boundedness of $A$,
\begin{multline*}
\iint_{\Omega_{\F_{N},\widetilde{Q}_0}^*} \big|\nabla (A^\top\,\nabla \G)(X)|^2\,\G_\top(X)\,dX
\\
\lesssim_\Lambda
\iint_{\Omega_{\F_{N},\widetilde{Q}_0}^*} |\nabla A(X)|^2 |\nabla \G(X)|^2\,\G_\top(X)\,dX
\\
+
\iint_{\Omega_{\F_{N},\widetilde{Q}_0}^*} |\nabla^2 \G(X)|^2\,\G_\top(X)\,dX
=:\I+\I\I.
\end{multline*}
The estimate for $\I$ is easy. Use \eqref{est:G:saw} and \eqref{tent-Q} to conclude as desired
$$
\I
\lesssim
\iint_{B_{\widetilde{Q}_0}^*\cap\Omega } |\nabla A(X)|^2 \,\delta(X)\,dX
\lesssim
\big\||\nabla A|\,\delta\big\|_{L^\infty(\Omega)}
\|\nabla A\|_{\C(\Omega)}\sigma(\Delta_{\widetilde{Q}_0}^*)
\approx
\big\||\nabla A|\,\delta\big\|_{L^\infty(\Omega)}
\|\nabla A\|_{\C(\Omega)}\sigma(\widetilde{Q}_0),
$$
where the implicit constants are clearly independent of $N$.

To estimate $\I\I$ we need the following auxiliary lemma whose proof will be postponed until the next the section.

\begin{lemma}\label{lemma:approx-saw}
There exists $\Psi_N\in C_0^\infty(\ree)$ such that

\begin{list}{$(\theenumi)$}{\usecounter{enumi}\leftmargin=.8cm
\labelwidth=.8cm\itemsep=0.2cm\topsep=.1cm
\renewcommand{\theenumi}{\roman{enumi}}}

\item $1_{\Omega_{\F_{N},\widetilde{Q}_0}^*}\lesssim \Psi_N\le 1_{\Omega_{\F_{N},\widetilde{Q}_0}^{**}}$.

\item $\sup_{X\in \Omega} |\nabla \Psi_N(X)|\,\delta(X)\lesssim 1$.

\item Set $\Sigma:=\partial\Omega_{\F_{N},\widetilde{Q}_0}^*$, 
\begin{equation}
\label{eq:defi-WN}
\W_N:=\bigcup_{Q\in\dd_{\F_{N},\widetilde{Q}_0}} \W_Q^*,
\qquad
\W_N^\Sigma:=
\big\{I\in \W_{N}:\, \exists\,J\in \W\setminus \W_N\quad \mbox{with}\quad  \partial I\cap\partial J\neq\emptyset
\big\}.
\end{equation}
\end{list}
Then
\begin{equation}
\nabla \Psi_N\equiv 0
\quad
\mbox{in}
\quad
\bigcup_{I\in \W_N \setminus \W_N^\Sigma }I^{***}
\quad\qquad\mbox{and}\qquad\quad
\sum_{I\in \W_N^\Sigma}\ell(I)^n\lesssim\sigma(\widetilde{Q}_0),
\label{eq:fregtgtr}
\end{equation}
with implicit constants depending on the allowable parameters but uniform in $N$.
\end{lemma}

Now we are ready to estimate $\I\I$. Using the previous lemma we have
$$
\I\I
\lesssim
\iint_{\ree} |\nabla^2 \G(X)|^2\,\G_\top(X)\,\Psi_N(X)\,dX
\le
\sum_{j=1}^{n+1} \iint_{\ree} |\nabla (\partial_j \G)(X)|^2\,\G_\top(X)\,\Psi_N(X)\,dX.
$$
Write ``$\partial$'' to denote a fixed generic derivative. The following observations will be used several times in the proof. Observe that Lemma \ref{lemma:est-nabla:II} and \eqref{est:G:saw} give that
 $\G, \G_\top, \nabla\G.\nabla\G_\top \in W^{1,2}(\overline{\Omega_{\star}})\cap L^\infty(\overline{\Omega_{\star}})$. Observe also that $\Psi_N$ is supported in $\Omega_\star$, thus $\partial\G\,\G_\top\,\Psi_N\in W^{1,2}_0(\Omega_\star)$. Hence we can find $\{\G_k\}_k\subset C_0^\infty(\Omega_\star)$ such that
$\G_k\to \partial\G\,\G_\top\,\Psi_N$ in $W^{1,2}(\Omega_\star)$. Note also that $|\nabla A|\in L^\infty(\overline{\Omega_{\star}})<\infty$ by our assumption $(a)$. These observations will, in particular, justify that all the integrals below are absolutely convergent.

We can now return to our task of estimating $\I\I$. By ellipticity and using
$\langle\cdot,\cdot\rangle$ to denote the inner product on $L^2(\ree)$ it follows that
\begin{multline*}
F_N:
=\iint_{\ree} |\nabla (\partial \G)(X)|^2\,\G_\top(X)\,\Psi_N(X)\,dX
\lesssim_\Lambda
\big\langle A^\top \nabla(\partial \G), \nabla(\partial \G)\,\G_\top\Psi_N \big\rangle
\\
=
\big\langle A^\top \nabla(\partial \G), \nabla(\partial \G\,\G_\top\Psi_N) \big\rangle
-
\frac12  \big\langle A^\top \nabla\big((\partial \G)^2\big),\nabla(\G_\top\Psi_N)\big\rangle
=:
\widetilde{\I}-\frac12 \widetilde{\I\I}.
\end{multline*}
To estimate $\widetilde{\I}$ we write
$$
\widetilde{\I}
=
\big\langle \partial( A^\top \nabla \G), \nabla(\partial \G\,\G_\top\Psi_N) \big\rangle
-
\big\langle \partial A^\top \nabla \G, \nabla(\partial \G\,\G_\top\Psi_N) \big\rangle
=:
\widetilde{\I}_1-\widetilde{\I}_2.
$$
Controlling $\widetilde{\I}_1$ it is not difficult as the previous observations along with \eqref{est:G:saw} give
$$
\widetilde{\I}_1
=
\lim_{k\to\infty} \big\langle \partial( A^\top \nabla \G), \nabla\G_k \big\rangle
=:
\lim_{k\to\infty} \widetilde{\I}_{1,k}.
$$
On the other hand
$$
\widetilde{\I}_{1,k}
=
\iint_{\ree}
\partial\Big( A^\top \nabla \G \cdot\nabla\G_k \Big)\,dX
-
\big\langle A^\top \nabla \G, \nabla\partial\G_k \rangle
=
\widetilde{\I}_{1,k,1}-\widetilde{\I}_{1,k,2}.
$$
Note that $A^\top \nabla \G \cdot\nabla\G_k\in W^{1,2}(\Omega_\star)$ and is supported in $\overline{\Omega_\star}$. Hence  $\widetilde{\I}_{1,k,1}=0$ by the divergence theorem. Also, $\widetilde{\I}_{1,k,2}=0$ since $L^\top \G=0$ in the weak-sense in $\Omega_{\star}$ (cf. \eqref{eq:normalization} and Lemma \ref{lemma2.green}) and $\partial \G_k\in C_0^\infty(\Omega_{\star})$. Therefore  $\widetilde{\I}_{1,k}=0$ and consequently $\widetilde{\I}_1=0$.

We next estimate $\widetilde{\I}_2$:
$$
\widetilde{\I}_2
=
\big\langle \partial A^\top \nabla \G, \nabla(\partial \G)\,\G_\top\Psi_N \big\rangle
+
\big\langle \partial A^\top \nabla \G, \nabla \G_\top\partial \G\,\Psi_N \big\rangle
+
\big\langle \partial A^\top \nabla \G, \nabla\Psi_N\,\partial \G\,\G_\top \big\rangle
=:
\widetilde{\I}_{21}+\widetilde{\I}_{22}+\widetilde{\I}_{23}.
$$
Note that by \eqref{est:G:saw}, Lemma \ref{lemma:approx-saw} and \eqref{tent-Q}
$$
|\widetilde{\I}_{22}|
\lesssim
\iint_{\ree} |\nabla A|\,|\nabla \G|^2|\nabla \G_\top|\,\Psi_N\,dX
\lesssim
\,\iint_{T_{\widetilde{Q}_0}^{**}}\,|\nabla A|dX
\lesssim
\|\nabla A\|_{\C(\Omega)}\,\sigma(\Delta_{\widetilde{Q}_0}^*)
\lesssim
\|\nabla A\|_{\C(\Omega)}\,\sigma(\widetilde{Q}_0),
$$
where we have used hypothesis $(c)$.
Also by \eqref{est:G:saw}, Lemma \ref{lemma:approx-saw} and \eqref{tent-Q},  and Young's inequality  we have
\begin{align*}
|\widetilde{\I}_{21}|
&\lesssim
\iint_{\ree} |\nabla A|\,|\nabla(\partial  \G)|\,\G_\top\,\Psi_N\,dX
\\
&\lesssim
\left(\iint_{\ree} |\nabla A|^2\,\G_\top\,\Psi_N\,dX\right)^{\frac12}
\left(\iint_{\ree} |\nabla(\partial  \G)|^2\,\G_\top\,\Psi_N dX\right)^{\frac12}
\\
&\lesssim
\left(\iint_{T_{\widetilde{Q}_0}^{**}} |\nabla A|^2\,\delta(\cdot)\, dX\right)^{\frac12}\,
F_N^{\frac12}
\\
&\lesssim
\big\||\nabla A|\,\delta\big\|_{L^\infty(\Omega)}^{\frac12}\,\|\nabla A\|_{\C(\Omega)}^{\frac12}
\sigma(\Delta_{\widetilde{Q}_0}^*)^{\frac12}
F_N^{\frac12}
\\
&\le
C\sigma(\widetilde{Q}_0)
+
\frac12\,F_N
.
\end{align*}
Note that $C$ depends on the 1-sided CAD constants, ellipticity,
$\big\||\nabla A|\,\delta\big\|_{L^\infty(\Omega)}$, $\|\nabla A\|_{\C(\Omega)}$
and $C_1$ fixed in the statement of Proposition \ref{prop:CME-G:ns}.
Analogously,
\begin{align*}
|\widetilde{\I}_{23}|
\lesssim
\iint_{\ree} |\nabla A|\, |\nabla\Psi_N|\,\delta(\cdot)\,dX
\lesssim
\iint_{T_{\widetilde{Q}_0}^{**}} |\nabla A|\,dX
\lesssim
\|\nabla A\|_{\C(\Omega)}\,\sigma(\widetilde{Q}_0).
\end{align*}
Collecting the obtained estimates we conclude that
$$
|\widetilde{\I}|=|\widetilde{\I}_2|\le
|\widetilde{\I}_{21}|+|\widetilde{\I}_{22}|+|\widetilde{\I}_{23}|
\le
C\,\sigma(\widetilde{Q}_0)
+
\frac12\,F_N.
$$

We next estimate $\widetilde{\I\I}$:
$$
\widetilde{\I\I}=
\big\langle A^\top \nabla\big((\partial \G)^2\big),\nabla \G_\top\Psi_N\big\rangle
+
\big\langle A^\top \nabla\big((\partial \G)^2\big),\nabla\Psi_N\,\G_\top\big\rangle
=: \widetilde{\I\I}_1+\widetilde{\I\I}_2.
$$
For $\widetilde{\I\I}_2$ we proceed as before, use  \eqref{est:G:saw}, Lemma \ref{lemma:approx-saw},  \eqref{eqn:Caccioppoli:nabla:general},  and Caccioppoli's  and Harnack's inequalities to obtain
\begin{multline*}
|\widetilde{\I\I}_2|
\lesssim
\iint_{\ree} |\nabla^2 \G|\,|\nabla\Psi_N|\,\delta(\cdot)\,dX
\lesssim
\sum_{I\in \W_N^\Sigma}\iint_{I^{***}}|\nabla^2 \G|\,dX
\lesssim
\sum_{I\in \W_N^\Sigma}\frac{|I|^{\frac12}}{\ell(I)}\left(\iint_{2\,I^{***}}|\nabla \G|^2\,dX\right)^{\frac12}
\\
\lesssim
\sum_{I\in \W_N^\Sigma}\frac{|I|^{\frac12}}{\ell(I)^2}\left(\iint_{3\,I^{***}}|\G|^2\,dX\right)^{\frac12}
\lesssim
\sum_{I\in \W_N^\Sigma}\frac{|I|}{\ell(I)}\,\frac{\G\big(X(I)\big)}{\delta(X(I))}
\lesssim
\sum_{I\in \W_N^\Sigma}\ell(I)^n
\lesssim
\sigma(\widetilde{Q}_0).
\end{multline*}
Let us turn our attention to $\widetilde{\I\I}_1$:
\begin{align*}
\widetilde{\I\I}_1=
\big\langle A\nabla \G_\top,
\nabla\big((\partial \G)^2\big)\Psi_N\big\rangle
=
\big\langle A\nabla \G_\top, \nabla\big((\partial \G)^2\,\Psi_N\big)\big\rangle
-
\big\langle
A\nabla \G_\top,
\nabla\Psi_N\,(\partial \G)^2\big\rangle
=:
\widetilde{\I\I}_{11}-\widetilde{\I\I}_{12}.
\end{align*}
Notice that $\widetilde{\I\I}_1=0$ since $L\G_\top$ in the weak sense in $\Omega_{\star}$
(cf. \eqref{eq:normalization-ns:top} and Lemma \ref{lemma2.green}) and
$(\partial \G)^2\,\Psi_N\in W^{1,2}_0(\Omega_{\star})$. Hence, another use of \eqref{est:G:saw} and Lemma \ref{lemma:approx-saw} produce
$$
|\widetilde{\I\I}_1|
=
|\widetilde{\I\I}_{12}|
\lesssim
\iint_{\ree} |\nabla \G|^2\,|\nabla \G_\top|\,|\nabla \Psi_N|\,dX
\lesssim
\sum_{I\in \W_N^\Sigma} \iint_{I^{***}}\,\frac1{\delta(X)}\,dX
\lesssim
\sum_{I\in \W_N^\Sigma} \ell(I)^n
\lesssim
\,\sigma(\widetilde{Q}_0).
$$
Putting things together
$$
|\widetilde{\I\I}|
\le
|\widetilde{\I\I}_{1}|
+
|\widetilde{\I\I}_{2}|
\lesssim\sigma(\widetilde{Q}_0).
$$

To conclude the proof we collect the obtained estimates
$$
0\le F_N
=
\widetilde{\I}-\frac12\,\widetilde{\I\I}
\le
|\widetilde{\I}|+\frac12\,|\widetilde{\I\I}|
\le
C\,\sigma(\widetilde{Q}_0)
+
\frac12\,F_N.
$$
Here all the constants are uniform in $N$. Since $F_N$ is finite
by \eqref{est:G:saw}, Lemma \ref{lemma:est-nabla:II} and the fact that $\supp \Psi_N\subset \overline{\Omega_\star}\subset \Omega$ we obtain
$$
F_N\lesssim
\,\sigma(\widetilde{Q}_0)
$$
which readily yields \eqref{eqn-goal:N:ns} with $C$ depending on the 1-sided CAD constants, ellipticity, $C_1$ fixed in the statement of Proposition \ref{prop:CME-G:ns}, $\big\||\nabla A|\,\delta\big\|_{L^\infty(\Omega)}$, and $\|\nabla A\|_{\C(\Omega)}$.
\end{proof}

\subsection{Proofs of Lemmas \ref{lemma:est-nabla:II} and \ref{lemma:approx-saw}}

In order to get the appropriate scale-invariant estimates in Lemma \ref{lemma:est-nabla:II}
we first present the case of the unit cube and then extend it to $\Omega$  by translation and rescaling.

\begin{lemma}\label{lemma:est-nabla:I}
Let $I_0:=(-\frac12,\frac12)^{n+1}\subset\ree$ and let $A\in \Lip(I_0)$ be a uniformly elliptic matrix in $I_0$. Given $K\ge 0$  there exists $C_{K}$ depending only on dimension, ellipticity and $K$  such that if $\|\nabla A\|_{L^\infty(I_0)}\le K$, then  for every $ u\in W^{1,2}(I_0)\cap L^\infty(I_0)$, $u\ge 0$, such that $Lu=0$ in the weak-sense in $I_0$ we have
\begin{equation}\label{eqn:point-nablaA:I0}
\sup_{X\in \frac12\,I_0} |\nabla u(X)|\le C_{K}\,\inf_{X\in \frac12 I_0} u(X),
\end{equation}
and
\begin{equation}\label{eqn:Caccioppoli:nabla:I0}
\iint_{\frac14\,I_0} |\nabla^2 u(X)|^2\,dX
\le
C_K\,\iint_{\frac12\,I_0} |\nabla u(X)|^2\,dX.
\end{equation}
\end{lemma}

\begin{proof}
To prove \eqref{eqn:point-nablaA:I0} we invoke  \cite[Lemma 3.1]{GW} in the open bounded domain $\frac34\,I_0$ and
there exist $C_K$ depending on $n$, ellipticity and $K$ such that
$$
\sup_{X\in \frac34I_0}|\nabla u(X)|\,\dist\big(X,\partial \big(\tfrac34I_0\big)\big)\le C_K\,\sup_{X\in \frac34 I_0}u(X).
$$
This and Harnack's inequality give at once \eqref{eqn:point-nablaA:I0}.

We next prove \eqref{eqn:Caccioppoli:nabla:I0}. Let us first observe that since $A$ is Lipschitz in $I_0$, and $u\in W^{1,2}(I_0)$ satisfies $Lu=0$ in the weak-sense in $I_0$, it follows that $u\in W^{2,2}(\frac34\,I_0)$ by \cite[Theorem 8.8]{GT}. With this in hand we are going to use a Caccioppoli type argument. Let $\varphi\in C_0^\infty(\ree)$  be a smooth cut-off of  $\frac14\,I_0$, that is, $1_{\frac14I_0}\le \varphi\le 1_{\frac38I_0}$ with $\|\nabla \varphi\|_{L^\infty}\le C_0$. Write ``$\partial$'' to denote a fixed generic derivative and observe that since $u\in W^{2,2}(\frac34\,I_0)$ it follows that $\partial u\,\varphi^2\in W^{1,2}_0(\frac12\,I_0)$. Hence there exists $\{u_k\}_k\subset C_0^\infty(\frac12\,I_0)$ such that
$u_k\to \partial u\,\varphi^2$ in $W^{1,2}(\frac12\,I_0)$. Writing $\Lambda$ for the ellipticity constant of $A$, we then have
\begin{multline*}
\I
:=
\iint_{\ree} |\nabla (\partial u)(X)|^2\,\varphi(X)^2\,dX
\le
\Lambda\,
\iint_{\ree} A(X)\,\nabla (\partial u)(X)\cdot \nabla(\partial u)(X)\,\varphi(X)^2\,dX
\\
=
\Lambda\Big(\iint_{\ree} A(X)\,\nabla (\partial u)(X)\cdot
\big[\nabla\big(\partial u\,\varphi^2\big)(X)- 2\nabla\varphi(X)\, (\partial u)(X)\,\varphi(X)\big]\,dX\Big)
=:\Lambda\,(\I_1-2\,\I_2).
\end{multline*}
For $\I_2$ we observe that by the Cauchy-Schwarz inequality
$$
|\I_2|=\Big|\iint_{\ree} A(X)\,\nabla (\partial u)(X)\cdot\nabla\varphi(X)\, (\partial u)(X)\,\varphi(X)\,dX\Big|
\le
\Lambda\,C_0\,\I^{\frac12}
\Big(\iint_{\frac12 I_0} |\nabla u(X)|^2\,dX\Big)^{\frac12}.
$$
For $\I_1$ we use the sequence $\{u_k\}_k$ introduced above and note that
\begin{align*}
\I_1^k:&=
\iint_{\ree} A(X)\,\nabla (\partial u)(X)\cdot \nabla u_k(X)\,dX
\\
&=
\iint_{\ree} \partial\big(A\,\nabla u\cdot \nabla  u_k\big)(X)\,dX
-
\iint_{\ree} A(X)\,\nabla u(X)\cdot \nabla(\partial u_k) (X)\,dX
\\
&\qquad\qquad
-
\iint_{\ree} \partial A(X)\,\nabla u(X)\cdot \nabla u_k(X)\,dX
\\
&
=
-
\iint_{\ree} \partial A(X)\,\nabla u(X)\cdot \nabla  u_k (X)\,dX.
\end{align*}
Here we have used that since $\{u_k\}\subset C_0^\infty(\frac12\,I_0)$ both terms in the second line vanish. In fact the first term is the integral of a derivative of a $W^{1,2}(\ree)$ compactly supported function, and the second term vanishes because $Lu=0$ in $I_0$ in the weak sense and $\partial u_k\in C_0^\infty(\frac12\,I_0)$. To continue with our estimate we observe that by Cauchy-Schwarz
\begin{multline*}
|\I_1|
=
\big|\lim_{k\to\infty} \I_1^k\big|
=
\Big|\iint_{\ree} \partial A(X)\,\nabla u(X)\cdot \nabla \big( \partial u\,\varphi^2\big)(X)\,dX\Big|
\\
\le
\|\nabla A\|_{L^\infty(I_0)}\,\left(
\iint_{\ree} |\nabla u(X)|\,|\nabla(\partial u)(X)|\,\varphi(X)^2\,dX
+
2\,\iint_{\ree} |\nabla u(X)|^2\,|\nabla\varphi(X)|\,\varphi(X)\,dX
 \right)
\\
\le
\|\nabla A\|_{L^\infty(I_0)}\left(\I^{\frac12}
\Big(\iint_{\frac12 I_0} |\nabla u(X)|^2\,dX\Big)^{\frac12}
+2\,C_0\,\iint_{\frac12 I_0} |\nabla u(X)|^2\,dX\right)
.
\end{multline*}
Collecting all the obtained estimates we conclude that
$$
\I
\le \Lambda\,\left(\|\nabla A\|_{L^\infty(I_0)} + 2\,C_0\Lambda\right)\I^{\frac12} \Big(\iint_{\frac12 I_0} |\nabla u(X)|^2\,dX\Big)^{\frac12}  +
2\,\Lambda\,C_0\, \|\nabla A\|_{L^\infty(I_0)} \,\iint_{\frac12 I_0} |\nabla u(X)|^2\,dX.
$$
From here we can use  Young's inequality with epsilon in the first term on the right hand side, hide $\I$ (which is finite since $u\in W^{2,2}(\frac34\,I_0)$) and the desired estimates follows easily.
\end{proof}

\begin{proof}[Proof of Lemma \ref{lemma:est-nabla:II}]

This result follows easily from Lemma \ref{lemma:est-nabla:I}. For $(i)$, first $u\in L^\infty_{\rm loc}(\Omega)$ by interior regularity. We take $J$,  any Whitney cube in $\Omega$, and translate and rescale $2\,J$ so that it becomes $\overline{I_0}$. Note that \eqref{eqn:point-nablaA:general} translates into the boundedness of the gradient of the corresponding matrix in Lemma \ref{lemma:est-nabla:I} (up to some dimensional constants). Hence \eqref{eqn:point-nablaA:I0} and Harnack's inequality give as desired  \eqref{eqn:Caccioppoli:general}.

The proof of \eqref{eqn:Caccioppoli:nabla:general}  follows easily from \eqref{eqn:Caccioppoli:nabla:I0}  by rescaling  and translation, again interior regularity gives that $u\in L^\infty_{\rm loc}(6I)$. Details are left to the reader.
\end{proof}
\begin{proof}[Proof of Lemma \ref{lemma:approx-saw}]
We recall that given $I$, any closed dyadic cube in $\ree$, we set $I^{**}=(1+2\,\lambda)I$ and $I^{***}=(1+4\,\lambda)I$. Let us introduce $\widetilde{I^{**}}=(1+3\,\lambda)I$ so that
\begin{equation}
I^{**}
\subsetneq
\interior(\widetilde{I^{**}})
\subsetneq \widetilde{I^{**}}
\subset
\interior(I^{***}).
\label{eq:56y6y6}
\end{equation}

Given $I_0:=[-\frac12,\frac12]^{n+1}\subset\ree$, fix $\phi_0\in C_0^\infty(\ree)$ such that
 $1_{I_0^{**}}\le \phi_0\le 1_{\widetilde{I_0^{**}}}$ and $|\nabla \phi_0|\lesssim 1$ (the implicit constant will depend on the parameter $\lambda$). For every $I\in \W=\W(\Omega)$ we set $\phi_I(\cdot)=\phi_0\big(\frac{\,\cdot\,-X(I)}{\ell(I)}\big)$ so that $\phi_I\in C^\infty(\ree)$, $1_{I^{**}}\le \phi_I\le 1_{\widetilde{I^{**}}}$ and $
|\nabla \phi_I|\lesssim \ell(I)^{-1}$ (with implicit constant depending only on $n$ and $\lambda$).

For every $X\in\Omega$, we let $\Phi(X):=\sum_{I\in \W} \phi_I(X)$. It then follows that $\Phi\in C_{\rm loc}^\infty(\Omega)$ since for every compact subset of $\Omega$, the previous sum has finitely many non-vanishing terms. Also, $1\le \Phi(X)\lesssim C_{\lambda}$ for every $X\in \Omega$ since the family $\{\widetilde{I^{**}}\}_{I\in \W}$ has bounded overlap by our choice of $\lambda$. Hence we
can set $\Phi_I=\phi_I/\Phi$ and one can easily see that $\Phi_I\in C_0^\infty(\ree)$, $C_\lambda^{-1}1_{I^{**}}\le \Phi_I\le 1_{\widetilde{I^{**}}}$ and $
|\nabla \Phi_I|\lesssim \ell(I)^{-1}$. With this in hand and by recalling the definition of $\W_N$ in \eqref{eq:defi-WN} we set
$$
\Psi_N(X)
:=
\sum_{I\in \W_N} \Phi_I(X)
=
\frac{\sum\limits_{I\in \W_N} \phi_I(X)}{\sum\limits_{I\in \W} \phi_I(X)},
\qquad
X\in\Omega.
$$
We first note that the number of terms in the sum defining $\Psi_N$ is bounded depending on $N$. Indeed, if $Q\in \dd_{\F_N, \widetilde{Q}_0}$ then $Q\in \dd_{\widetilde{Q}_0}$ and $2^{-N}\ell(\widetilde{Q}_0)<\ell(Q)\le \ell(\widetilde{Q}_0)$ which implies that $\dd_{\F_N, \widetilde{Q}_0}$ has finite cardinality with bounds depending only on the AR property and $N$. Also, by construction $\W_Q^*$ has cardinality depending only in the allowable parameters. Hence, $\# \W_N\lesssim C_N<\infty$. This and the fact that each $\Phi_I\in C_0^\infty(\ree)$ yield that $\Psi_N\in C_0^\infty(\ree)$. Note also that \eqref{eq:56y6y6} and the definition of $\W_N$ in \eqref{eq:defi-WN} give
$$
\supp \Psi_I
\subset
\bigcup_{I\in \W_N} \widetilde{I^{**}}
=
\bigcup_{Q\in\dd_{\F_{N},\widetilde{Q}_0}}
\bigcup_{I\in \W_Q^*} \widetilde{I^{**}}
\subset
\interior\Big(
\bigcup_{Q\in\dd_{\F_{N},\widetilde{Q}_0}}
\bigcup_{I\in \W_Q^*} I^{***}
\Big)
=
\interior\Big(
\bigcup_{Q\in\dd_{\F_{N},\widetilde{Q}_0}}
U_Q^{**}
\Big)
=
\Omega_{\F_{N},\widetilde{Q}_0}^{**}.
$$
This, the fact that $\W_N\subset \W$ and the definition of $\Psi_N$ immediately gives that
$\Psi_N\le 1_{\Omega_{\F_{N},\widetilde{Q}_0}^{**}}$. On the other hand if $X\in \Omega_{\F_{N},\widetilde{Q}_0}^{*}$ then the exists $I\in \W_N$ such that $X\in I^{**}$ in which case $\Psi_N(X)\ge \Phi_I(X)\ge C_\lambda^{-1}$. This completes the proof of $(i)$.

To obtain $(ii)$ we note that for every $X\in \Omega$
$$
|\nabla \Psi_N(X)|
\le
\sum_{I\in \W_N} |\nabla\Phi_I(X)|
\lesssim
\sum_{I\in \W} \ell(I)^{-1}\,1_{\widetilde{I^{**}}}(X)
\lesssim
\delta(X)^{-1}
$$
where we have used that if $X\in \widetilde{I^{**}}$ then $\delta(X)\approx \ell(I)$ and also that the family $\{\widetilde{I^{**}}\}_{I\in \W}$ has bounded overlap.

Let us finally address $(iii)$. Fix $I\in\W_N\setminus \W^{\Sigma}_N$ and  $X\in I^{***}$, and set $\W_X:=\{J\in \W: \phi_J(X)\neq 0\}$. We first note that  $\W_X\subset \W_N$. Indeed, if $\phi_J(X)\neq 0$ then $X\in \widetilde{J^{**}}$.
Hence $X\in I^{***}\cap J^{***}$ and our choice of $\lambda$ gives that $\partial I$ meets $\partial J$, this in turn implies that $J\in \W_N$ since $I\in\W_N\setminus \W^{\Sigma}_N$. All these yield
$$
\Psi_N(X)
=
\frac{\sum\limits_{J\in \W_N} \phi_J(X)}{\sum\limits_{J\in \W} \phi_J(X)}
=
\frac{\sum\limits_{J\in \W_N\cap \W_X} \phi_J(X)}{\sum\limits_{J\in \W\cap \W_X} \phi_J(X)}
=
\frac{\sum\limits_{J\in \W_N\cap \W_X} \phi_J(X)}{\sum\limits_{J\in \W_N\cap \W_X} \phi_J(X)}
=
1.
$$
Hence $\Psi_N\big|_{I^{***}}\equiv 1$  for every $I\in\W_N\setminus \W^{\Sigma}_N$. This and the fact that $\Psi_N\in C_0^\infty(\ree)$ immediately give that $\nabla \Psi_N\equiv 0$ in $\bigcup_{I\in \W_N \setminus \W_N^\Sigma }I^{***}$.

To complete the proof we need to estimate the sum in \eqref{eq:fregtgtr}. Recall that $\Sigma=\partial \Omega_{\F_{N},\widetilde{Q}_0}^{*}\subset \Omega$ and let $I\in \W_N^\Sigma$. We claim that there exists $Z_I\in \Sigma$ such that $\dist(Z_I,I)\approx \ell(I)\approx \delta(Z_I)$. To prove this we first observe that $\interior(I^{**})\subset \Omega_{\F_{N},\widetilde{Q}_0}^{*}\subset \Omega$ since $I\in \W_N$. On the other hand, $I\in \W_N^\Sigma$ implies that there is  $J\in \W\setminus \W_N$ such that $\partial I\cap \partial J\neq\emptyset$. In particular, $X(J)\in \ree\setminus \Omega_{\F_{N},\widetilde{Q}_0}^{*}$ (by our choice of $\lambda$) where $X(J)$ is the center of $J$. Then we can find $Z_I\in \Sigma$ with $Z_I$ in the segment joining $X(J)$ and $X(I)$. Note that $\dist(Z_I, I)\le |Z_I-X(I)|\le |X(I)-X(J)|\lesssim \ell(I)$ since $\partial I\cap \partial J\neq\emptyset$ implies that $\ell(I)\approx \ell(J)$ by the nature of the Whitney cubes. On the other hand since $\interior(I^{**})\subset \Omega_{\F_{N},\widetilde{Q}_0}^{*}$ we have that  $Z_I\notin \interior(I^{**})$, thus $\dist(Z_I, I)\gtrsim \ell(I)$ (with implicit constant depending on $\lambda$). Finally, $\ell(I)\approx\ell(J)\approx\delta(Z_I)$.

One we have chosen $Z_I$ we let $\Delta_I^\Sigma=B(Z_I,\delta(Z_I)/2)\cap \Sigma$, which is a surface ball with respect to the domain $\Omega_{\F_{N},\widetilde{Q}_0}^{*}$ centered on $Z_I\in \Sigma=\partial \Omega_{\F_{N},\widetilde{Q}_0}^{*}$.  Since $\partial \Omega_{\F_{N},\widetilde{Q}_0}^{*}$ is AR (cf. \cite[Lemma 3.61]{HM-URHM}) with bounds that do not depend on $N$,  it follows that
$$
\sum_{I\in \W_N^\Sigma}\ell(I)^n
\approx
\sum_{I\in \W_N^\Sigma}\delta (Z_I)^n
\approx
\sum_{I\in \W_N^\Sigma}
H^n(\Delta_I^\Sigma).
$$
We next see that the family $\{\Delta_I^\Sigma\}_{I\in \W_N^\Sigma}$ has bounded overlap. Indeed, suppose that $\Delta_{I_1}^\Sigma\cap \Delta_{I_2}^\Sigma\neq\emptyset$ and take $Y$ in that intersection. Assume for instance that $\ell(I_1)\le \ell(I_2)$. then,
$$
\delta(Z_{I_2})
\le |Z_{I_2}-Y|+|Y-Z_{I_1}|+\delta(Z_{I_1})
\le
\frac12\,\delta(Z_{I_2}) +\frac32\, \delta(Z_{I_1})
$$
which implies that $\ell(I_2)\approx\delta(Z_{I_2})\lesssim \delta(Z_{I_1})\approx \ell(I_1)$. Thus, $\ell(I_1) \approx \ell(I_2)$. Moreover,
$$
\dist(I_1, I_2)
\le
\dist(I_1, Z_{I_1})
+
|Z_{I_1}-Y|
+
|Y-Z_{I_2}|
+
\dist(I_2, Z_{I_2})
\lesssim
\ell(I_1)+\ell(I_2)
\approx
\ell(I_1)
\approx
\ell(I_2).
$$
By the properties of the Whitney cubes it then follows that the family $\{\Delta_I^\Sigma\}_{I\in \W_N^\Sigma}$ has bounded overlap. Thus,
$$
\sum_{I\in \W_N^\Sigma}\ell(I)^n
\approx
\sum_{I\in \W_N^\Sigma}
H^n(\Delta_I^\Sigma)
\lesssim
H^n\Big(\bigcup_{I\in \W_N^\Sigma}\Delta_I^\Sigma\Big)
\le
H^n(\Sigma)
=
H^n(\partial \Omega_{\F_{N},\widetilde{Q}_0}^{*})
\lesssim
\diam(\partial \Omega_{\F_{N},\widetilde{Q}_0}^{*})^n
\lesssim
\ell(\widetilde{Q}_0)^n,
$$
where we have used again that $\partial \Omega_{\F_{N},\widetilde{Q}_0}^{*}$ is AR and also that this set is bounded with diameter controlled by $\ell(\widetilde{Q}_0)$. This completes the proof of Lemma \ref{lemma:approx-saw}.
\end{proof}

\appendix

\section{The $A_\infty$ property in Lipschitz domains:  the Kenig-Pipher argument}\label{KP}

The result of \cite{KKiPT} allows for a slight condensation  of the proof of the results of \cite{KP},
albeit with the very same ideas.  For the reader's convenience, we supply the shortened proof here
following the key part of \cite{KP} essentially unchanged. 
To be precise, we shall prove the following. 

\begin{knowntheorem} {\bf (\cite{KP})}.
Let $\Omega\subset \mathbb{R}^{n+1}$ be a Lipschitz domain, and suppose that $L=-\div A \nabla$
is an elliptic operator in $\Omega$ satisfying Hypothesis \ref{hyp1}, but
with property $(c)$ replaced by the weaker condition
\eqref{car-A-square}.  Then elliptic measure is absolutely continuous with respect to surface measure $\sigma$ on $\pom$, and the Poisson kernel satisfies Hypothesis \ref{hyp2}, with constants depending only on dimension, the Lipschitz character of $\om$, and the constants in the modified version of Hypothesis \ref{hyp1} that we assume here.
\end{knowntheorem}

\begin{proof}[Sketch of Proof]
Since the estimate to be proved, namely \eqref{eq:higher-inte}, is local, we may reduce matters to working in
a single co-ordinate patch, and thus we may suppose that $\Omega=\{(x,t)\in \mathbb{R}^{n+1}:  \, t>\varphi(x)\}$,
where $\varphi$ is a Lipschitz function.  We may then further suppose that $\Omega=\reu$, the upper half-space,
by pulling back under an appropriate mapping (see, e.g., \cite{DKPV}) which preserves the class of
coefficients satisfying the modified
Hypothesis \ref{hyp1}
(i.e., with property $(c)$ replaced by
\eqref{car-A-square}).  By \cite{KKiPT}, we may further reduce matters to proving the Carleson measure estimate
\begin{equation}\label{A1}
\sup_Q  \frac1{|Q|} \int_0^{\ell(Q)}\!\!\!\int_Q |\nabla u(x,t)|^2 \, t\, dx dt \,\leq\, C \|u\|_\infty^2 \,,
\end{equation}
for any bounded weak non-negative solution of the equation $Lu=0$ in $\reu$, and the
supremum runs over all cubes $Q\subset \mathbb{R}^n$.  At this point we follow the argument of
\cite{KP} essentially verbatim.

Fix $u\ge 0$ a bounded weak solution of the equation $Lu=0$ in $\reu$.  Note that by the preceding reductions, $A=(a_{i,j})_{1\leq i,j\leq n+1}$ satisfies the modified Hypothesis
\ref{hyp1} in $\reu$, with $X:=(x,t) \in \mathbb{R}^n\times (0,\infty)$, and $\delta(x,t) = t$.
In particular, by property $(b)$, which now becomes $|\nabla A(x,t)|\lesssim 1/t$,
we have that $|\nabla u(x,t)|\lesssim  t^{-1}\,\|u\|_\infty$, uniformly in $x$ (see \eqref{eqn:Caccioppoli:general}).

Observe that if we set $A':= \big(a_{n+1,n+1}\big)^{-1} A$ (note that $a_{n+1,n+1}\ge \Lambda^{-1}>0$ by ellipticity),  then
$$L'u =-\div A'\nabla u = -\frac1{a_{n+1,n+1}} Lu  -  \nabla \left(\frac1{a_{n+1,n+1}}\right)\cdot A\nabla u
= -  \nabla \left(\frac1{a_{n+1,n+1}}\right)\cdot A\nabla u\,,$$
since $Lu=0$;  i.e.,
$L'u +{\bf B}\cdot\nabla u = 0,$
where
${\bf B} = (B_1,B_2,...,B_{n+1})$, with
$$B_k =\sum_{j=1}^{n+1}\frac{\partial}{\partial_{X_j}}  \left(\frac1{a_{n+1,n+1}}\right) a_{j,k}\,,
\qquad 1\leq k\leq n+1\,,$$
and $X_{n+1}=t$.  Then by our current assumptions on $A$,  $|{\bf B}|\lesssim 1/t$, and
$|{\bf B}|^2 t dx dt $ is a Carleson measure in $\reu$.

Thus, after relabeling $A', L'$ as $A,L$, and normalizing so that $\|u\|_\infty \leq 1$,
we may suppose that
\begin{equation}\label{A2}
a_{n+1,n+1} =1\,, \qquad 
Lu + {\bf B}\cdot \nabla u = 0\,, \qquad \|u\|_\infty + t \|\nabla u(\cdot,t)\|_\infty \lesssim 1\,,
\end{equation}
where  $L=-\div A\nabla$, $|{\bf B}|\lesssim 1/t$, and
$|{\bf B}|^2 t dx dt $ is a Carleson measure.

Fix a cube $Q\subset \re^n$, we define standard and two-sided Carleson boxes respectively, by
$$
R_Q:= Q\times \big(0,\ell(Q) \big)\,,\qquad R^*_Q:= Q\times \big(-\ell(Q),\ell(Q) \big)\,.
$$
Let $\Phi=\Phi_Q\in C_0^\infty(\re^{n+1})$ be a smooth cut-off adapted to $R_Q$, so that
$\supp \Phi \subset R_{2Q}^*,$ $\Phi\equiv 1$ in $R_Q^*$, $0\leq \Phi\leq 1$, and
$\|\nabla \Phi\|_\infty \lesssim 1/\ell(Q)$.

Set
$$d\mu(x,t):= |\nabla A(x,t)|^2\, t \,dx dt \,,\qquad    d\nu(x,t):= |{\bf B}(x,t)|^2 \,t \,dx dt \,,$$
and define their respective Carleson norms by
\begin{equation}
\|\mu\|_\cc\,:= \,\sup_Q\,\frac{\mu(R_Q)}{|Q|}\,,\qquad \|\nu\|_\cc\,:= \,\sup_Q\,\frac{\nu(R_Q)}{|Q|}\,.
\label{eq:A-carleson-app}
\end{equation}

To prove the corresponding estimate in \eqref{A1} for $Q$, it is routine to see that we can work with $u_\eta, A_\eta$ and ${\bf B}_\eta$,
in place of $u$, $A$ and ${\bf B}$,  defined by $u_\eta(x,t):= u(x,t+\eta)$, etc., and then let $\eta \to 0^+$ provided all our estimates are independent of $\eta$. To simplify the presentation we abuse the notation and use $u$, $A$ and ${\bf B}$ to denote respectively $u_\eta, A_\eta$ and ${\bf B}_\eta$. Notice that \eqref{A2} remains true
with bounds uniform in $\eta$, and also that $u=u_\eta$ is
continuous in $\overline{\ree_+}$.   We use ellipticity and then
the second equation in \eqref{A2} to write
\begin{align}\label{A3}
&\iint_{R_Q}|\nabla u|^2 \, t\, dx dt\, \leq \, \iint_{\reu}|\nabla u|^2 \, \Phi \,t\, dx dt
\\[4pt] \nonumber
&\lesssim
\iint_{\reu}\langle A\nabla u,\nabla u\rangle\, \Phi \, t\, dx dt
\\[4pt]\nonumber
&=\, \iint_{\reu}\left\langle A\nabla u,\nabla \big(u \, \Phi \, t\big)\right\rangle\, dx dt\,-
\, \iint_{\reu}\left\langle A\nabla u,\nabla \big( \Phi \, t\big)\right\rangle\, u\,dx dt
\\[4pt] \nonumber
&=\, -\,\iint_{\reu}{\bf B}\cdot\nabla u \, u\,\Phi \,t \,dx dt
 \,-\, \iint_{\reu}\left\langle A\nabla u,\nabla\Phi\right\rangle\, u\,t\,dx dt
  \,-\, \iint_{\reu}\left\langle A\nabla u,e_{n+1}\right\rangle\, u\,\Phi\,dx dt
	\\[4pt] \nonumber
&  =:  -\I_1-\I_2-\I_3,
\end{align}
where $e_{n+1}$ denotes the standard unit basis vector in the positive $t$ direction.

We first treat $\I_2$.  By the last item in \eqref{A2}, and the construction of $\Phi$,
we find that
$$|\I_2| \lesssim \frac1{\ell(Q)}\iint_{R_{2Q}}
1\,dx dt \lesssim |Q|\,.$$

Next, we consider $\I_3$,
which we rewrite as
$$
\I_3
=  \sum_{j=1}^n\iint_{\reu} a_{n+1,j}\,\big(\partial_j u\big) \, u\,\Phi\,dx dt\, + \,
\iint_{\reu}  \big(\partial_t u\big) \, u\,\Phi\,dx dt \,=: \,\I\I + \I\I\I\,,
$$
since we have reduced to the case that $a_{n+1,n+1}\equiv 1$.
Then, since $u$ is continuous in $\overline{\ree_+}$,
$$
\I\I\I = \frac12 \iint_{\reu} \partial_t \big( u^2\big)\,\Phi\,dx dt
\,=\, - \frac12 \iint_{\reu} \big(\partial_t\Phi\big)\,u^2\,dx dt -\frac12\int_{\rn} u^2 \,\Phi \,dx\,,
$$
whence it follows that $|\I\I\I|\lesssim |Q|$, by the properties of $\Phi$, and the normalization $\|u\|_\infty\leq 1$.
We also have
\begin{align*}
\I\I
&= \sum_{j=1}^n\frac12 \iint_{\reu} a_{n+1,j}\,\partial_j \big(u^2\big) \, \Phi\,dx dt
\\[4pt]
&=
 -\frac12   \sum_{j=1}^n\iint_{\reu}  \partial_t
\left(a_{n+1,j}\,\partial_j\big(u^2\big) \, \Phi\right) t\,dx dt
\\[4pt]
&= -\frac12 \sum_{j=1}^n  \iint_{\reu} \partial_t\big(a_{n+1,j}\big)\,\partial_j \big(u^2\big) \, \Phi\,t\, dx dt\,
 -\frac12 \sum_{j=1}^n \iint_{\reu} a_{n+1,j}\,\partial_t\left(\partial_j \big(u^2\big)\right) \, \Phi\,t\, dx dt
\\[4pt]
&\qquad\qquad  -\frac12 \sum_{j=1}^n \iint_{\reu} a_{n+1,j}\,\partial_j \big(u^2\big) \, \big(\partial_t\Phi\big)\,t\, dx dt
\\[4pt]
&=:\,  -\frac12 \sum_{j=1}^n
\left(\I\I_{j,1} + \I\I_{j,2} + \I\I_{j,3}\right)\,,
\end{align*}
where we have integrated by parts in $t$ in the second line.
Exactly as for term $\I_2$, we find that $|\I\I_{j,3}|\lesssim |Q|$, for each $j$.
Integrating by parts horizontally, we find that
$$
\I\I_{j,2} =  - \iint_{\reu} \partial_j\big(a_{n+1,j}\big)\,\partial_t \big(u^2\big) \, \Phi\,t\, dx dt
-  \iint_{\reu} a_{n+1,j}\,\partial_t \big(u^2\big)\, \big(\partial_j\Phi\big)\,t\, dx dt
=:\I\I_{j,2}'+\I\I_{j,2}''\,.
$$
Note that $|\I\I_{j,2}''|\lesssim |Q|$, for each $j$, exactly as for term $\I_2$.

It remains to treat the terms $\I_1$, $\I\I_{j,1}$, and $\I\I_{j,2}'$, for which we have the cumulative estimate
\begin{multline*}
|\I_1| +|\I\I_{j,1}| +|\I\I_{j,2}'| \lesssim \iint_{\reu}\big(|{\bf B}| +|\nabla A|\big)\,|\nabla u| \, |u|\,\Phi \,t \,dx dt
\\[4pt]
\lesssim\, \frac1{\eps}\big(\mu(R_{2Q})+\nu(R_{2Q})\Big)
+ \, \eps \iint_{\reu}|\nabla u|^2 \, \Phi \,t\, dx dt\,,
\end{multline*}
where $\eps$ is at our disposal, and where we have used the definition of $\Phi$ and the
normalization $\|u\|_\infty\leq 1$.   Choosing $\eps$ small enough,
we may then hide
the small term on the left hand side (more precisely in the second term) in \eqref{A3};
note that this is finite since we are working with $u_\eta, A_\eta$ and ${\bf B}_\eta$. Also, by taking $0<\eta\le \ell(Q)$, clearly
$\mu(R_{2Q})+\nu(R_{2Q})\lesssim (\|\mu\|_\cc+\|\nu\|_\cc)\,|Q|$ uniformly on $\eta$. Collecting our various estimates, letting $\eta\to 0^+$ and since $Q$ was arbitrary,
we find that \eqref{A1} holds with $C \approx \|\mu\|_\cc + \|\nu\|_\cc$.
\end{proof}

\end{document}